\documentclass[12pt]{article}
\usepackage[usenames,dvipsnames]{color}
\usepackage{amsmath,amsthm,verbatim,amssymb,amsfonts,amscd,diagrams, graphics, mathrsfs,hyperref}

\theoremstyle{plain}
\newtheorem{theorem}{Theorem}[section]
\newtheorem{corollary}[theorem]{Corollary}
\newtheorem{lemma}[theorem]{Lemma}
\newtheorem{proposition}[theorem]{Proposition}

\theoremstyle{definition}
\newtheorem{definition}[theorem]{Definition}

\newtheorem{notation}[theorem]{Notation}

\theoremstyle{remark}
\newtheorem{remark}[theorem]{Remark}
\newtheorem{example}[theorem]{Example}

\newarrow{ul}---->
\newarrow{Backwards}<----   

\usepackage[matrix,arrow,curve,frame]{xy}

\newcommand{\into}{\hookrightarrow}
\newcommand{\Z}{\mathbb{Z}}
\newcommand{\Q}{\mathbb{Q}}
\newcommand{\R}{\mathbb{R}}

\newcommand{\bd}{\partial}

\renewcommand{\H}{\mathbb H}

\newcommand{\mc}[1]{\mathcal{#1}}
\newcommand{\ms}[1]{\mathscr{#1}}

\newcommand{\Hom}{\text{Hom}}

\newcommand{\codim}{\text{codim}}
\newcommand{\xr}{\xrightarrow}

\newcommand{\im}{\text{im}}

\newarrow{Onto}----{>>}
\newarrow{Equals}=====
\newarrow{Into}C--->

\newcommand{\capp}{\smallfrown}
\newcommand{\cC}{{\cal C}}
\newcommand{\dlim}{\mathop{\varinjlim}}

\begin{document}

\title{Cup and cap products in intersection (co)homology}
\author{Greg Friedman\thanks{This work was partially supported by a grant 
from the Simons Foundation (\#209127 to Greg Friedman)}\ \  and James 
McClure\thanks{The second author was partially supported by NSF. He
thanks the Lord for making his work possible.}
}

\date{December 1, 2012}

\maketitle

\begin{abstract}
We construct cup and cap products in intersection (co)homology with field
coefficients. The existence of the cap product allows us to give a new proof
of Poincar\'e duality in intersection (co)homology which is similar in spirit
to the usual proof for ordinary (co)homology of manifolds.
\end{abstract}



\tableofcontents

\section{Introduction}
Intersection homology, first introduced and studied by Goresky and MacPherson \cite{GM1,GM2}, is a basic tool in the study of singular spaces.  It has
important features in common with ordinary homology (excision, 
Mayer-Vietoris, intersection pairing, Poincar\'e duality with field
coefficients) and important differences (restrictions on functoriality, 
failure of homotopy invariance, restrictions on Poincar\'e duality with integer
coefficients).  An important difference is the fact 
that the Alexander-Whitney map does not induce a map of intersection chains 
(because if a simplex satisfies the relevant allowability condition there is
no reason for its front and back faces to do so).  Because of this, it has 
long been thought that there is no reasonable way to define cup and cap 
products in intersection (co)homology.  In this paper, we use a different 
method to construct cup and cap products (with field coefficients) with the 
usual properties, and we use the cap product to give a new proof of 
Poincar\'e duality for intersection (co)homology with field coefficients.

We give applications and extensions of these results in \cite{GBF30} and 
\cite{GBF31}.  In \cite{GBF30} we show that our Poincar\'e duality isomorphism 
agrees with that obtained by sheaf-theoretic methods in \cite{GM2}
and that our cup product is Poincar\'e dual to the intersection pairing of
\cite{GM2}.  We also prove that the de Rham isomorphism of \cite{BHS} takes 
the wedge product of intersection differential forms to the cup product of 
intersection cochains.  In \cite{GBF31} we give a new construction of the 
symmetric signature for Witt spaces (which responds to a question raised in 
\cite{ALMP-combo}).  

\begin{remark}
It is important to note that the cup product given by our work is not a map
\[
I_{\bar p}H^*(X;F)
\otimes
I_{\bar p}H^*(X;F)
\to
I_{\bar p}H^*(X;F)
\]
for a fixed perversity $\bar p$.  Instead it is a map
\[
I_{\bar p}H^*(X;F)
\otimes
I_{\bar q}H^*(X;F)
\to
I_{\bar r}H^*(X;F)
\]
for triples $\bar p, \bar q,\bar r$ satisfying a certain inequality (see 
Section \ref{S: cup}; this is what one would expect if the cup product is to be
dual to the intersection pairing).  In
particular, there is no relation between our work and the 
observation of Goresky and MacPherson (\cite[Section 6.2]{GM2}) that 
there seems to be no good cup product on middle-perversity intersection
cohomology. 
\end{remark}

Our basic strategy for constructing cup and cap products is to replace the
Alexander-Whitney map 
by a combination of the cross product and the (geometric) diagonal
map.  To illustrate this, we explain how it works 
in ordinary homology.  For a field $F$, the cross product gives an 
isomorphism (where
the tensor is over $F$)
\[
H_*(X;F)\otimes H_*(Y;F)\to H_*(X\times Y;F),
\]
and we can use this isomorphism to construct the algebraic diagonal map 
\[
\bar{d}:H_*(X;F)\to H_*(X;F)\otimes H_*(X;F)
\]
as the composite
\[
H_*(X;F)\xrightarrow{d} H_*(X\times X;F)
\xleftarrow{\cong} H_*(X;F)\otimes H_*(X;F),
\]
where $d$ is the geometric diagonal.  The evaluation map induces an
isomorphism
\[
H^*(X;F)\to \Hom(H_*(X;F),F),
\]
and we 
define the cup product of cohomology classes $\alpha$ and $\beta$ by
\[
(\alpha\cup\beta)(x)=(\alpha\otimes\beta)(\bar{d}(x)),
\quad x\in H_*(X; F).
\]
The fact that the cup product is associative,
commutative, and unital follows easily from the corresponding properties of the
cross product.  Similarly, we define the cap product by
\[
\alpha\cap x=(1 \otimes \alpha)(\bar d(x)).
\]

\bigskip

In order to carry out the analogous constructions in intersection homology, we 
need to know that the cross product gives an isomorphism on intersection 
homology (with suitable perversities) and that the geometric diagonal map 
induces a map of intersection homology (with suitable perversities).  The 
first of these facts is Theorem \ref{j1}
and the second is Proposition \ref{L: diag}.

\bigskip

Here is an outline of the paper.  In Section \ref{S: background}, we establish 
terminology and notations for stratified pseudomanifolds and intersection 
homology.  (We allow strata of codimension one and completely general 
perversities, which means that intersection homology is not independent of the 
stratification in general.)  In Section \ref{S: kunneth}, we state the 
K\"unneth theorem for intersection homology, which is the basic tool in our 
work.  In Section \ref{j3}, we construct the algebraic diagonal map, cup 
product, and cap product in intersection (co)homology and show that they have 
the expected properties.  In Section \ref{S: fund}, we show that an 
orientation of an $n$-dimensional stratified pseudomanifold $X$ determines a 
fundamental class in $I^{\bar 0}H_n(X,X-K;R)$ for each compact $K$ and each 
ring $R$.  In Section \ref{j4}, we show that cap product with the fundamental 
class induces a Poincar\'e duality isomorphism
\[
I_{\bar p}H^i_c(X;F)
\to
I^{\bar q}H_{n-i}(X;F)
\]
when $\bar p$ and $\bar q$ are complementary perversities.
In Section \ref{S: boundary}, we extend our results to
stratified pseudomanifolds with boundary. The proofs in Sections \ref{S: fund}
through \ref{S: boundary} follow the general outline of the corresponding proofs in
\cite{Ha}, but the details are more intricate.

\begin{remark}[Signs]
We include a sign in the Poincar\'e duality isomorphism (see \cite[Section 
4.1]{GBF18}).  Except for this we follow the 
signs in \cite{Dold}, which means that we use the Koszul convention 
everywhere except in the definition of the coboundary on cochains (see Remark 
\ref{R: coboundary sign} below).
\end{remark}

\begin{remark}
Recently, Chataur, Saralegi-Aranguren and Tanr\'e
\cite{CST}
have found a different approach to the cup product in intersection cohomology
which leads to an ``intersection'' version of rational homotopy theory.  It
seems likely that their cup product agrees with ours.  They do not construct a
cap product.
\end{remark}


\begin{remark}
\label{l5}
In \cite[Section 7]{Ba10a}, Markus Banagl constructed a cup product 
\[
I_{\bar p}H^*(X;\mathbb Q)
\otimes
I_{\bar p}H^*(X;\mathbb Q)
\to
I_{\bar q}H^*(X;\mathbb Q)
\]
for certain pairs of perversities $\bar p,\bar q$ (namely for classical
perversities satisfying
$\bar{p}(k)+\bar{p}(l)\leq\bar{p}(k+l)\leq\bar{p}(k)+\bar{p}(l)+2$ for all
$k,l$ and $\bar{q}(k)+k\leq \bar{p}(2k)$ for all $k$).  We show in Appendix 
\ref{l6} that this cup product agrees with ours (up to sign) for all such 
pairs $\bar p,\bar q$.  Banagl's construction is similar to ours, except that 
the K\"unneth theorem he uses is the one in \cite{CGJ} (which is a special 
case of that in \cite{GBF20}; see \cite[Corollary 3.6]{GBF20}).  He does not 
consider the cap product.
\end{remark}

\section{Background}\label{S: background}

We begin with a brief review of basic definitions. Subsection \ref{S: pm}
reviews the definition of stratified pseudomanifold. 
Subsection \ref{j5} reviews singular intersection homology with general 
perversities as 
defined in \cite{GBF26, GBF23}.  Other standard sources for more classical 
versions of intersection homology include  \cite{GM1, GM2, Bo, KirWoo, BaIH, 
Ki, GBF10}.

\subsection{Stratified pseudomanifolds}\label{S: pm}

We use the definition of stratified pseudomanifold in \cite{GM2}, except that
we allow strata of codimension one. 
Before giving the definition we need some background.

For a space $W$ we define the {\it open cone} $c(W)$ by $c(W)=([0,1)\times 
W)/(0\times W)$ (we put the $[0,1)$ factor first so that our signs will 
be consistent with the usual definition of the algebraic mapping cone). Note 
that $c(\emptyset)$ is a point.

If $Y$ is a filtered space
$$Y=Y^n\supseteq Y^{n-1}\supseteq \cdots \supseteq 
Y^0\supseteq Y^{-1}=\emptyset,$$ 
we define $c(Y)$ to be the filtered space with
$(c(Y))^i=c(Y^{i-1})$ for $i\geq 0$ and $(c(Y))^{-1}=\emptyset$.

The definition of stratified pseudomanifold is now given by induction on the
dimension. 

\begin{definition}\label{D: pseudomanifold}
A $0$-dimensional stratified pseudomanifold $X$ is a  discrete set of points 
with the trivial filtration $X=X^0\supseteq X^{-1}=\emptyset$.

An $n$-dimensional \emph{(topological) stratified  pseudomanifold}
$X$ is a paracompact Hausdorff space together 
with a filtration by closed subsets

\begin{equation*}
X=X^n\supseteq X^{n-1} \supseteq X^{n-2}\supseteq \cdots \supseteq X^0\supseteq X^{-1}=\emptyset
\end{equation*}
such that
\begin{enumerate}
\item $X-X^{n-1}$ is dense in $X$, and
\item for each point $x\in X^i-X^{i-1}$, there exists a neighborhood
$U$ of $x$ for which there is a  \emph{compact} $n-i-1$ dimensional 
stratified    pseudomanifold  $L$ and a   homeomorphism
\begin{equation*}
\phi: \R^i\times cL\to U
\end{equation*}
that takes $\R^i\times c(L^{j-1})$ onto $X^{i+j}\cap U$. A neighborhood $U$ with
this property is called {\it distinguished} and $L$ is called a {\it link} of
$x$.
\end{enumerate}
\end{definition}

The $X^i$ are called \emph{skeleta}. We write $X_i$ for $X^i-X^{i-1}$; this 
is an $i$-manifold that may be empty. We refer to the connected components of 
the various $X_i$ as  \emph{strata}\footnote{This terminology agrees with some 
sources, but is slightly different from others, including our own past work, 
which would refer to $X_i$ as the stratum and what we call strata as 
``stratum components.''}. If a stratum $Z$ is a subset of $X_n$ it 
is called a \emph{regular stratum}; otherwise it is called a \emph{singular 
stratum}.  The \emph{depth} of a stratified pseudomanifold is the number of 
distinct skeleta it possesses minus one.

We note that this definition of stratified pseudomanifolds is slightly more general than the one in common usage \cite{GM1}, as it is common to assume that $X^{n-1}=X^{n-2}$. We will not make that assumption here, but when we do assume $X^{n-1}=X^{n-2}$, intersection homology with Goresky-MacPherson perversities is known to be a topological invariant; in particular, it is invariant under choice of  stratification (see \cite{GM2}, \cite{Bo}, \cite{Ki}). Examples of pseudomanifolds include irreducible complex algebraic and analytic varieties (see \cite[Section IV]{Bo}).

If $L$ and $L'$ are links of points in the same stratum then there is a stratified  homotopy equivalence between them (see, e.g., \cite{GBF13}), 
and therefore they have the same intersection homology by Appendix \ref{A: sphe}.  Because of this, we
will sometimes refer to ``the link'' of a stratum  instead of ``a link'' of  
a point in the stratum.

\subsection{Singular intersection homology with general perversities.}
\label{j5}

\begin{definition}
Let $X$ be a stratified pseudomanifold.
A  \emph{perversity on $X$} is a
function $\bar p: \{\text{strata of $X$} \}\to \Z$
with $\bar p(Z)=0$ if $Z$ is a regular stratum.   
\end{definition}

This is a much more general definition than that in
\cite{GM1,GM2}; on the rare occasions when we want to refer to perversities as
defined in \cite{GM1,GM2} we will call them ``classical perversities.''

Besides being interesting in their own right, general perversities are required
in our work because of their role in the K\"unneth theorem (Theorem \ref{j1}
below).

In the literature, there are several non-equivalent definitions of 
intersection homology with general perversities.  We will use the version in
\cite{GBF26, GBF23} (which is equivalent to that in \cite{Sa05}).  The reason
for this choice is that it gives the most useful version of the ``cone
formula'' (Proposition \ref{P: cone} below).

As motivation for the definition, recall that the singular chain group 
$S_i(X; G)$ of a space $X$ with coefficients in an abelian group $G$ consists 
of finite sums $\sum g_j\sigma_j$, where each $g_j\in G$ and each $\sigma_j$ 
is a map $\sigma_j:\Delta^i\to X$ of the standard $i$-simplex to $X$. The 
boundary is given by $\bd \sum g_j\sigma_j=\sum g_j\bd \sigma_j=\sum_{j,k} 
(-1)^kg_j\bd_k\sigma_k$. If instead $\mc G$ is a local coefficient system of 
abelian groups, then an element of $S_i(X; \mc G)$ is again a sum $\sum 
g_j\sigma_j$, where now $g_j$ is a lift of $\sigma_j$ to $\mc G$ or, 
equivalently, a section of the coefficient system $\sigma_j^*\mc G$ over 
$\Delta^i$. The boundary map becomes $\bd \sum g_j\sigma_j=\sum_{j,k} 
(-1)^kg_j|_{\bd_k\sigma_j}\bd_k \sigma_j$; in other words, the restriction of 
the ``coefficient'' $g_j$ to the boundary piece $\bd_k\sigma_j$ is the 
restriction of the section over $\Delta^i$ to $\bd_k\Delta^i$. If the system 
$\mc G$ is constant, then we recover $S_i(X; G)$.

If $X$ is a stratified pseudomanifold we make a slight adjustment. Suppose $\mc G$ is a local coefficient system defined on $X-X^{n-1} $. Let
$C_i(X;\mc G)$ again consist of chains $\sum g_j\sigma_j$, where now $g_j$ is a section of 
$(\sigma_j|_{\sigma_j^{-1}(X-X^{n-1})})^*\mc G$  over $\sigma^{-1}_j(X-X^{n-1})$.  Note that if $\sigma_j^{-1}(X-X^{n-1})$ is empty then the sections of $(\sigma_j|_{\sigma_j^{-1}(X-X^{n-1})})^*\mc G$ form the trivial group (because
there's exactly one map from the empty set to any set). The
differential is given by the same formula as in the previous paragraph, with  
restrictions  to boundaries $\bd_k\Delta^i$ being trivial if $\sigma_j$ maps $\bd_k\Delta^i$ into $X^{n-1}$. Even when we have a globally defined coefficient system, such at the constant system $G$, we continue to let\footnotemark $C_i(X; \mc G)$ denote $C_i(X;\mc G|_{X-X^{n-1}})$. 
\footnotetext{In the first-named  author's prior work, this would have been denoted $C_i(X; \mc G_0)$. }

Now given a stratified pseudomanifold $X$, a general perversity $\bar p$, and 
a local coefficient system $\mc G$ on $X-X^{n-1} $, we define the 
\emph{intersection chain complex} $I^{\bar p}C_*(X;\mc G)$ as a subcomplex of 
$C_*(X;\mc G)$ as follows. An $i$-simplex $\sigma:\Delta^i\to X$ in $C_i(X)$ 
is \emph{allowable} if $$\sigma^{-1}(Z)\subset \{i-\codim(Z)+\bar p(Z) \text{ 
skeleton of } \Delta^i\}$$ for any singular stratum  $Z$ of $X$. The chain 
$\xi\in C_i(X;\mc G)$ is allowable if each simplex with non-zero coefficient 
in $\xi$ or in $\bd \xi$ is allowable. $I^{\bar p}C_*(X;\mc G)$ is the 
complex of allowable chains. 

The associated homology theory is denoted $I^{\bar p}H_*(X;\mc G)$ 
and called \emph{intersection homology}.  Relative intersection homology is 
defined similarly. 
 
If $\bar p$ is a perversity in the sense of Goresky-MacPherson \cite{GM1} and 
$X$ has no strata of codimension one, then $I^{\bar p}H_*(X;\mc G)$ is 
isomorphic to the intersection homology groups $I^{\bar p}H_*(X;\mc G)$ of 
Goresky-MacPherson \cite{GM1, GM2}.

Even with general perversities, many of the basic properties of singular
intersection homology established in  \cite{Ki} and \cite{GBF10} hold with 
little or no change to the proofs, such as excision and Mayer-Vietoris sequences. Intersection homology is also invariant under properly formulated stratified versions of homotopy equivalences. Proof of this folk result for Goresky-MacPherson perversities is written down in \cite{GBF3}; the slightly more elaborate details necessary for general perversities are provided below in Appendix \ref{A: sphe}.

Intersection homology with general perversities can also be formulated sheaf 
theoretically; see \cite{GBF23, GBF26} for more details.

\paragraph{Cone formula.}
General perversity intersection homology satisfies the following cone 
formula, which generalizes that in \cite{GM1,GM2} (but it differs from King's 
formula in \cite{Ki}); see \cite[Proposition 2.1]{GBF23} and 
\cite[Proposition 2.18]{GBF10}.  We state it with constant 
coefficients, which is all that we require. 

\begin{proposition}\label{P: cone}
Let $L$ be an $n-1$ dimensional stratified pseudomanifold, and let $G$ be an
abelian group. Let $cL$ be the cone on $L$, with
vertex $v$ and stratified so that $(cL)^0=v$ and $(cL)^i=c(L^{i-1})$ for 
$i>0$. Then
\begin{equation*}
I^{\bar p}H_i(cL;G) \cong
\begin{cases}
0, & i\geq n-1-\bar p(\{v\}),\\
I^{\bar p}H_{i}(L;G), & i<n-1-\bar p(\{v\}),
\end{cases}
\end{equation*}
where the isomorphism in the second case is induced by any inclusion 
$\{t\}\times L\into ([0,1)\times 
L)/(0\times L)=cL$ with $t\neq 0$. 

Therefore, also
\begin{equation*}
I^{\bar p}H_i(cL,L;G)\cong 
\begin{cases}
I^{\bar p}H_{i-1}(L;G), & i\geq n-\bar p(\{v\}),\\
0, & i<n-\bar p(\{v\}).
\end{cases}
\end{equation*}
\end{proposition}

\paragraph{A word about notation.}\label{notation}
The reader might be concerned that we are using the notation $I^{\bar p}H_*(X;\mc G)$ in a manner that is not consistent with all previous authors. There are, however, compelling reasons to consider our definition as the ``correct'' one for singular intersection homology in most settings. We enumerate these here:

\begin{enumerate}
\item For perversities $\bar p$ such that\footnote{Here $\bar t$ is the \emph{top perversity}, for which $\bar t(S)=\codim(S)-2$ if $S$ is a singular stratum.} $\bar p(S)\leq \bar t(S)$ for all singular strata $S$ and such that $\bar p$ depends only on the codimensions of $S$, our $I^{\bar p}H_*(X;\mc G)$ agrees with the original PL intersect homology of Goresky-MacPherson \cite{GM1} (who only considered perversities satisfying these conditions) on PL stratified pseudomanifolds, and it also agrees with the singular intersection homology of King \cite{Ki} (who only considers perversities satisfying the second condition); see \cite[Proposition 2.1]{GBF10}. It follows that our definition of intersection homology yields groups isomorphic to \emph{all} of those considered by Goresky and MacPherson in \cite{GM1} and \cite{GM2} either by geometric chain methods or sheaf-theoretic methods. In fact, the only disagreement between our groups and those appearing in the works of Goresky-MacPherson or  King that employ geometric chain methods occurs in the cases for which $\bar p(S)> \bar t(S)$ for some strata, for which our groups disagree with those of King. 

\item As a consequence of Proposition \ref{P: cone}, it follows that our intersection homology groups agree with those arising from the hypercohomology of the Deligne sheaf in situations where King's do not. In particular, our intersection homology groups agree with the sheaf-theoretic groups of Goresky-MacPherson \cite{GM2}, as well as  those of Cappell and Shaneson \cite{CS91} who require perversities such that $\bar p(S)$ might be $>\bar t(S)$. These works of Goresky-MacPherson and Cappell-Shaneson all use perversities such that $\bar p(S)\leq \bar p(T)$ if $T$ has a greater codimension than $S$, but the Deligne sheaf can be modified as in \cite{GBF23} to acquire sheaf theoretic generalizations of intersection homology duality via Verdier duality, and it is the geometric intersection homology groups presented here that agree with the hypercohomology of these modified Deligne sheaves\footnote{In \cite{HS91}, Habegger and Saper work out a different sheaf theoretic construction whose hypercohomology yields King's singular sheaves. However, as we discuss in \cite{GBF26}, this construction is quite complex, does not apply to the most general perversities, and yields a duality theorem that is not quite as satisfying as the one presented here (in particular, duality only holds using quite complex sheaf theoretic coefficients).}. An exposition of this background material can be found in \cite{GBF26}. 

\item The agreement of our intersection homology groups with those arising in sheaf theory settings that permit the use of Verdier duality, together with the work presented in this paper, demonstrates that our definition of intersection homology is suitable for obtaining intersection homology duality results. By contrast, working with the intersection homology groups of King (or with a more direct generalization of the definition of Goresky and MacPherson in the PL setting), Poincar\'e duality fails when perversities that take values above those of $\bar t$ are involved. For example, let us use the formulas in \cite{Ki} to compute King's intersection homology groups for the suspended torus $ST$ with the perversity $\bar p$ defined so that it takes the value $2$ at each of the suspensions points. As the suspension points have codimension $3$, $\bar p$ is greater than $\bar t$. The dual perversity $\bar q$ in this case takes the value $-1$ at each of the suspension points. Then, denoting the King intersection homology groups by $K^{\bar p}H_*$, we have

\begin{equation*}
K^{\bar p}H_i(ST;\Q)\cong 
\begin{cases}
\Q, & i=3,\\
\Q\oplus \Q,& i=2,\\
0, &i=1,\\
\Q, & i=0,
\end{cases} 
\qquad 
K^{\bar q}H_i(ST;\Q)\cong 
\begin{cases}
0, & i=3,\\
\Q,& i=2,\\
\Q\oplus \Q, &i=1,\\
\Q, & i=0.
\end{cases}
\end{equation*} 
We see that the  duality theorem must fail for these groups.

\item The intersection homology groups utilized here agree with those of Saralegi in his de Rham Theorem for unfoldable pseudomanifolds, relating singular intersection homology to perverse differential forms \cite{Sa05}. A proof of the agreement between our definition and Saralegi's can be found in the appendix to \cite{GBF20}. 
\end{enumerate}

Given these points, it seems reasonable to the authors to use the notation $I^{\bar p}H_*(X;\mc G)$ and the phrase ``intersection homology'' to refer to the groups described here.

\section{The K\"unneth theorem for intersection homology}\label{S: kunneth}

Let $X$ and $Y$ be stratified pseudomanifolds, and let $F$ be a field.  We 
stratify
$X\times Y$ in the obvious way: for any strata $Z\subset X$ and $S\subset Y$,
$Z\times S$
is a stratum of $X\times Y$.

By \cite[page 382]{GBF20}, the cross product (where
the tensor is over $F$)
\[
C_*(X;F)\otimes C_*(Y;F)\to C_*(X\times Y;F)
\]
restricts to give a map
\[
I^{\bar p}C_*(X;F)\otimes I^{\bar q}C_*(Y;F)\to I^QC_*(X\times Y;F)
\]
provided that $Q(Z\times S)\geq {\bar p}(Z)+{\bar q}(S)$ for all strata
$Z\subset X$, $S\subset Y$.  

We can now state the K\"unneth theorem:

\begin{theorem}
\label{j1}
Let $\bar p$ and $\bar q$ be perversities on $X$ and $Y$, and define a
perversity $Q_{\bar p,\bar q}$ on $X\times Y$ by
\begin{equation*}
 Q_{\bar p,\bar q}(Z\times S)=\begin{cases}
 \bar p(Z)+\bar q(S)+2, &\text{$Z,S$ both singular strata,}\\
\bar p(Z), &\text{$S$ a regular stratum and $Z$ singular,}\\
\bar q(S), &\text{$Z$ a regular stratum and $S$ singular,}\\
0,&\text{$Z,S$ both regular strata.}
 \end{cases}
 \end{equation*}
Then the cross product induces an isomorphism
\[
I^{\bar p}H_*(X;F)\otimes I^{\bar q}H_*(Y;F)\to I^QH_*(X\times Y;F).
\]
\end{theorem}

This is a somewhat sharper form of the main result of \cite{GBF20}.
We show in Appendix \ref{Appendix A} how to deduce it from the results of
\cite{GBF20}.

\begin{remark}
In fact there are other choices of $Q$ that give isomorphisms, as 
explained in \cite{GBF20}, but this is the right choice for our purposes 
because of its compatibility with the diagonal map; see Proposition 
\ref{L: diag}.
\end{remark}

There is also a relative version of the K\"unneth theorem. 

\begin{theorem}\label{P: relative kunneth}
Let $X$ and $Y$ be stratified pseudomanifolds with open subsets $A\subset X, B\subset Y$. 
The cross product induces an isomorphism 
\[
I^{\bar p}H_*(X,A;F)\otimes I^{\bar 
q}H_*(Y,B;F)\to I^{Q_{\bar p,\bar q}}H_*(X\times Y,(A\times Y)\cup 
(X\times B);F). 
\]
\end{theorem}

The  proof is given in Appendix \ref{Appendix A}.

\section{The diagonal map,  cup product, and cap product}
\label{j3}

\subsection{The diagonal map}\label{S: diag}

In this subsection, we define the algebraic diagonal map using the method 
described
in the introduction.  The first step is to show that the geometric diagonal map
induces a map of intersection chains for suitable perversities.  

First we need some notation.  Recall that 
the {\it top perversity} $\bar t$ is defined by 
\[
{\bar t}(Z)
=
\begin{cases}
0, &  \text{if $Z$ is regular}, \\
\codim(Z)-2, & \text{if $Z$ is singular}.
\end{cases}
\]

\begin{definition}\label{D: dual perv}
Let $\bar p$ be a perversity. Define the \emph{dual perversity} $D\bar p$ by
\[
D\bar p(Z)=\bar t(Z)-\bar p(Z).
\]
\end{definition}

With the notation of Theorem \ref{j1},
let us write $Q_{\bar p,\bar q,\bar r}$ for $Q_{Q_{\bar p,\bar q},\bar r}$
(which is equal to $Q_{\bar p,Q_{\bar q,\bar r}}$).

\begin{proposition}\label{L: diag}
Let $d:X\to X\times X$ be the diagonal and let $G$ be an abelian group.
\begin{enumerate}
\item 
If
$D\bar r(Z)\geq D\bar p(Z)+D\bar q(Z)$ for each  stratum $Z$ of $X$ then 
$d$ induces a map
\[
d: I^{\bar r}C_*(X;G)\to I^{Q_{\bar p,\bar q}}C_*(X\times X;G).
\]

\item 
If 
$D\bar s(Z)\geq D\bar q(Z)+D\bar r(Z)$ for each stratum
$Z$ of $X$ then $1\times d$ induces a map
\[
1\times d: I^{Q_{\bar p,\bar s}}C_*(X\times X;G)\to I^{Q_{\bar 
p,\bar q,\bar r}}C_*(X\times X\times X;G). 
\]

\item 
If $D\bar s(Z)\geq D\bar p(Z)+D\bar q(Z)$ for 
each stratum
$Z$ of $X$, then $d\times 1$ induces a map
\[
d\times 1: I^{Q_{\bar s,\bar r}}C_*(X\times X;G)\to I^{Q_{\bar 
p,\bar q,\bar r}}C_*(X\times X\times X;G). 
\]
\end{enumerate}
\end{proposition}

\begin{proof}
We prove the first part, the other two are similar.
A chain $\xi$ is in $I^{\bar r}C_i(X;G)$ if for any 
simplex $\sigma$ of $\xi$ and any singular stratum $Z$ of $X$, 
$\sigma^{-1}(Z)$ is contained in the $i-\codim(Z)+\bar r(Z)$ skeleton of the 
model simplex $\Delta^i$. 
Now the only singular strata of $X\times X$ which intersect the
image of $d$ have the form
$Z\times Z$, where $Z$ is a singular stratum of $X$,
so the chain $d(\xi)$ will be in $I^{Q_{\bar 
p,\bar q}}C_i(X\times X;G)$ if each $(d\sigma)^{-1}(Z\times Z)$ is 
contained in the $i-\codim(Z\times Z)+Q_{\bar p,\bar q}(Z\times Z)$ 
skeleton of the model simplex $\Delta^i$. 
For this it suffices to have
\[
i-\codim(Z)+\bar r(Z)\leq i-\codim(Z\times Z)
+Q_{\bar p,\bar q}(Z\times Z)=
i-2\,\codim(Z)+\bar p(Z)+\bar q(Z)+2,
\]
and this is equivalent to the condition in the hypothesis.
\end{proof}

Now we can define the algebraic diagonal map.

\begin{definition}\label{D: algdiag}
If $D\bar r\geq \bar Dp+\bar Dq$ let 
\[
\bar d: I^{\bar r}H_*(X;F)\to I^{\bar p}H_*(X;F)\otimes I^{\bar 
q}H_*(X;F)
\]
be the composite
\[
I^{\bar r}H_*(X;F) \xrightarrow{d}
I^{Q_{\bar p,\bar q}}H_*(X\times X;F)
\xleftarrow{\cong}
I^{\bar p}H_*(X;F)\otimes I^{\bar
q}H_*(X;F),
\]
where the second map is the K\"unneth isomorphism (Theorem \ref{j1}).
\end{definition}

In the remainder of this subsection we show that the algebraic diagonal map has 
the expected properties.

Note that $\bar d$ is a natural map due to the naturality of the cross product. 

\begin{proposition}[Coassociativity]\label{C: coass}
Suppose that 
$D\bar s\geq D\bar u+D\bar r$, 
$D\bar s\geq D\bar p+D\bar v$, 
$D\bar u\geq D\bar p+D\bar q$ and
$D\bar v\geq D\bar q+D\bar r$. 
Then the following diagram commutes 
\begin{diagram}
I^{\bar s}H_*(X;F) &\rTo^{\bar d}& I^{\bar u}H_*(X;F)\otimes I^{\bar r}H_*(X;F)\\ 
\dTo^{\bar d}&&\dTo_{\bar d\otimes 1}\\
I^{\bar p}H_*(X;F)\otimes I^{\bar v}H_*(X;F)&\rTo_{1\otimes \bar d}&I^{\bar p}H_*(X;F)\otimes I^{\bar q}H_*(X;F)\otimes I^{\bar r}H_*(X;F).
\end{diagram}
\end{proposition}

\begin{proof}
Consider the following diagram (with coefficients left tacit):
\[
\xymatrix{
I^{\bar s}C_*(X)
\ar[r]^-d
\ar[d]_d
&
I^{Q_{\bar u,\bar r}}C_*(X\times X)
\ar[d]^{d\times 1}
&
I^{\bar u}C_*(X)\otimes I^{\bar r}C_*(X)
\ar[l]_-{\mathrm{q.i.}}
\ar[d]^{d\otimes 1}
\\
I^{Q_{\bar p,\bar v}}C_*(X\times X)
\ar[r]^-{1\times d}
&
I^{Q_{\bar p,\bar q,\bar r}}C_*(X\times X\times X)
&
I^{Q_{\bar p,\bar q}}C_*(X\times X)\otimes I^{\bar r}C_*(X)
\ar[l]_-{\mathrm{q.i.}}
\\
I^{\bar p}C_*(X)\otimes I^{\bar v}C_*(X)
\ar[u]^{\mathrm{q.i.}}
\ar[r]^-{1\otimes d}
&
I^{\bar p}C_*(X)\otimes I^{Q_{\bar q,\bar r}}C_*(X\times X)
\ar[u]_{\mathrm{q.i.}}
&
I^{\bar p}C_*(X)\otimes I^{\bar q}C_*(X)\otimes I^{\bar r}C_*(X)
\ar[l]_-{\mathrm{q.i.}}
\ar[u]_{\mathrm{q.i.}}
}
\]
Here the arrows $1\times d$ and $d\times 1$
exist by parts 2 and 3 of Proposition \ref{L: diag}. The arrows marked q.i.\ are induced by the cross product and are
quasi-isomorphisms by Theorem \ref{j1}.  The upper left square obviously
commutes, the upper right and lower left squares commute by naturality of the
cross product, and the lower right square commutes by associativity of the
cross product.  The result follows from this.
\end{proof}

\begin{proposition}[Cocommutativity]\label{C: cocomm}
If $D\bar r\geq D\bar p+D\bar q$ then the following diagram commutes.
\begin{diagram}
I^{\bar r}H_*(X;F)&\rTo^{\bar d} & I^{\bar q}H_*(X;F)\otimes I^{\bar p}H_*(X;F)\\
&\rdTo^{\bar d}&\dTo^{\cong}\\
&&I^{\bar p}H_*(X;F)\otimes I^{\bar q}H_*(X;F).
\end{diagram}
\qed
\end{proposition}

As background for our next result, note that for any $\bar q$ and any abelian 
group $G$ there is an augmentation $\varepsilon:I^{\bar q}H_*(X;G)\to G$ 
that takes a 0-chain to the sum of its coefficients and all other chains to 
$0$.  Also note that $D\bar t$ is identically $0$, so for every $\bar p$
there is an algebraic diagonal map 
\[
\bar d: I^{\bar p}H_*(X;F)\to I^{\bar t}H_*(X;F)\otimes I^{\bar
p}H_*(X;F).
\]

\begin{proposition}[Counital property]\label{P: counital}
For any $\bar p$, the composite
\[
I^{\bar p}H_*(X;F) \xrightarrow{\bar d}
I^{\bar t}H_*(X;F)\otimes I^{\bar
p}H_*(X;F)
\xrightarrow{\varepsilon \otimes 1}
F\otimes I^{\bar p}H_*(X;F)
\cong
I^{\bar p}H_*(X;F)
\]
is the identity.
\end{proposition}

\begin{proof}
First observe that (by an easy argument using 
the definition of allowable chain) the projection $p_2:X\times X\to X$ 
induces a map
\[
I^{Q_{\bar t,\bar p}} H_*(X\times X; F)\to 
I^{\bar p} H_*(X;F). 
\]
Now it suffices to observe that
the following diagram commutes. 
\[
\xymatrix{
I^{\bar p}H_*(X;F) 
\ar[r]^-d
\ar[dr]_=
&
I^{Q_{\bar t,\bar p}}H_*(X\times X)
\ar[d]_{p_2}
&
I^{\bar t}H_*(X;F)\otimes I^{\bar p}H_*(X;F)
\ar[l]_-\cong
\ar[d]_{\varepsilon\otimes 1}
\\
&
I^{\bar p}H_*(X;F)
&
F\otimes I^{\bar p}H_*(X;F)
\ar[l]_-\cong
}
\]
The commutativity of the square follows easily
from the fact that the cross product is induced by the chain-level shuffle 
product \cite[Exercise VI.12.26(2)]{Dold}.\end{proof}

The results of this subsection also have relative forms. Suppose $A$ and $B$ 
are 
open subsets of $X$ and that $D\bar r\geq D\bar p+D\bar q$. Then there is an 
algebraic diagonal 
\[
\bar d: I^{\bar r}H_*(X, A\cup B;F)\to I^{\bar p}H_*(X, A;F)\otimes I^{\bar 
q}H_*(X, B;F),
\]
and the obvious generalizations of the preceding results hold.  Moreover, 
we have the following proposition.

\begin{proposition}
\label{j2}
Let $A$ be an open subset of $X$ and let $i:A\to X$ be the inclusion.
Then
the diagram
\[
\xymatrix{
I^{\bar r}H_*(X, A;F)
\ar[rr]^-{\bar d}
\ar[d]_\partial
&&
I^{\bar p}H_*(X, A;F)\otimes I^{\bar q}H_*(X;F)
\ar[d]^{\partial\otimes 1}
\\
I^{\bar r}H_*(A;F)
\ar[r]^-{\bar d}
&
I^{\bar p}H_*(A;F)\otimes I^{\bar q}H_*(A;F)
\ar[r]^-{1\otimes i}
&
I^{\bar p}H_*(A;F)\otimes I^{\bar q}H_*(X;F)
}
\]
commutes.
\end{proposition}

\begin{proof}
This follows from the commutativity of the diagrams
\[
\xymatrix{
I^{\bar r}H_*(X, A;F)
\ar[rr]^-d
\ar[d]_\partial
&&
I^{Q_{\bar p,\bar q}}H_*(X\times X,A\times X;F)
\ar[d]^\partial
\\
I^{\bar r}H_*(A;F)
\ar[r]^-d
&
I^{Q_{\bar p,\bar q}}H_*(A\times A;F)
\ar[r]^-{1\times i}
&
I^{Q_{\bar p,\bar q}}H_*(A\times X;F),
}
\]
which commutes by the naturality of $\bd$,

\[
\xymatrix{
I^{Q_{\bar p,\bar q}}H_*(X\times X,A\times X;F)
\ar[d]_\partial
&
I^{\bar p}H_*(X, A;F)\otimes I^{\bar q}H_*(X;F)
\ar[l]_-\cong
\ar[d]^{\partial\otimes 1}
\\
I^{Q_{\bar p,\bar q}}H_*(A\times X;F)
&
I^{\bar p}H_*(A;F)\otimes I^{\bar q}H_*(X;F),
\ar[l]_-\cong
}
\]
which commutes because the cross product is a chain map, and

\[
\xymatrix{
I^{Q_{\bar p,\bar q}}H_*(A\times A;F)
\ar[r]^-{1\times i}
&
I^{Q_{\bar p,\bar q}}H_*(A\times X;F)
& 
I^{\bar p}H_*(A;F)\otimes I^{\bar q}H_*(X;F)\ar[l]_-\cong
\\
&
I^{\bar p}H_*(A;F)\otimes I^{\bar q}H_*(A;F),
\ar[lu]^\cong
\ar[ru]_{1\otimes i}
&
}
\]
which commutes by naturality of the cross product.
\end{proof}

\subsection{Cochains and the cup product}\label{S: cup}

We begin by defining intersection cochains and intersection cohomology with
field coefficients.

\begin{definition}
Define $I_{\bar p}C^*(X;F)$ to be
$\Hom_F(I^{\bar p}C_*(X;F),F)$ and $I_{\bar p}H^*(X;F)$ to be 
$H^*( I_{\bar p}C^*(X;F))$. 
Similarly for the relative groups:
$I_{\bar p}C^*(X,A;F)$ is $\Hom_F(I^{\bar p}C_*(X,A;F),F)$ and  $I_{\bar 
p}H^*(X,A;F)$ is $H^*( I_{\bar p}C^*(X,A;F))$.
\end{definition}

\begin{remark}
\label{l1}
Because $F$ is a field we have
\[
I_{\bar p}H^*(X;F)\cong\Hom_F(I^{\bar p}H_*(X;F),F)
\]
 and 
\[
I_{\bar p}H^*(X,A;F)\cong\Hom_F(I^{\bar p}H_*(X,A;F),F). 
\]
\end{remark}

\begin{remark}\label{R: coboundary sign} We will typically write $\alpha$ for a cochain and $x$ for a
chain.
We follow Dold's convention for the differential of a cochain
(see \cite[Remark VI.10.28]{Dold}; note that this differs slightly from the
Koszul convention):
\[
(\delta\alpha)(x)=-(-1)^{|\alpha|}\alpha(\bd x).
\]
This convention is necessary in order for the evaluation map to be a chain map.
\end{remark}

\begin{definition}
\label{l7}
If $D\bar s\geq D\bar p+D\bar q$,
we define the {\em cup product} in intersection cohomology $$\smallsmile: 
I_{\bar 
p}H^*(X;F)\otimes I_{\bar q}H^*(X;F)\to I_{\bar s}H^*(X;F)$$ by
$$(\alpha\smallsmile \beta)(x)=(\alpha\otimes \beta)\bar d(x).$$
Explicitly, if $d(x)=\sum y_i\times z_i$, then 
$(\alpha\smallsmile\beta)(x)=\sum (-1)^{|\beta||y_i|} \alpha(y_i)\beta(z_i)$. 
\end{definition}

\begin{proposition}[Associativity]\label{L: assoc}
Let $\bar p,\bar q,\bar r,\bar s$ be perversities such that $D\bar s \geq 
D\bar p+D\bar q+D\bar r$. Let $\alpha\in I_{\bar p}H^*(X;F)$, $\beta\in 
I_{\bar p}H^*(X;F)$, and
$\gamma\in I_{\bar r}H^*(X;F)$.
Then
\[
(\alpha\smallsmile\beta)\smallsmile\gamma
=\alpha\smallsmile(\beta\smallsmile\gamma)
\]
in $I_{\bar s}H^*(X;F)$. 
\end{proposition}

\begin{proof}
This follows from Proposition \ref{C: coass}, with 
$\bar u=D(D\bar p+D\bar q)$ and $\bar v=D(D\bar q+D\bar r)$.
\end{proof}

\begin{proposition}[Commutativity]\label{L: comm}
Let $\bar p,\bar q, \bar s$ be perversities such that $D\bar s \geq D\bar 
p+D\bar q$.
Let $\alpha\in I_{\bar p}H^*(X;F)$, $\beta\in I_{\bar q}H^*(X;F)$. Then 
\[
\alpha\smallsmile \beta 
= (-1)^{|\alpha||\beta|} \beta\smallsmile \alpha
\]
in $I_{\bar s}H^*(X;F)$.
\end{proposition}

\begin{proof}
This is immediate from Proposition \ref{C: cocomm}.
\end{proof}

\subsection{The cap product}\label{S: cap}

\begin{definition}
If $D\bar r\geq D\bar p+D\bar q$ and $A$, $B$ are open subsets of $X$, we
define the {\em cap product}
$$\smallfrown:I_{\bar q}H^i(X,B ;F)\otimes I^{\bar r}H_j(X, A\cup B;F)\to 
I^{\bar p}H_{j-i}(X, A;F)$$
by
\[
\alpha\smallfrown x=(1\otimes \alpha)\bar{d}(x).
\]
Explicitly, if $d(x)=\sum y_i\times z_i$, then $\alpha\smallfrown x=\sum
(-1)^{|\alpha||y_i|} \alpha(z_i)y_i$.
\end{definition}

\begin{remark}
This definition is modeled on \cite[Section VII.12]{Dold}.  The reason 
Dold has $1\otimes \alpha$ instead of $\alpha\otimes 1$ in the definition is
so that the cap product will make the chains a left module over the cochains
(in accordance with the fact that $\alpha$ is on the left in the symbol
$\alpha\smallfrown x$).
\end{remark}

In the remainder of this subsection we show that the cap product has the
expected properties.  We begin with the analogue of \cite[VII.12.6]{Dold}.

\begin{proposition}\label{P: cap inclusion}
Suppose $D\bar r\geq D\bar p+D\bar q$. Let $A,B,X', A',B'$ be open subsets of
$X$ with $A'\subset X'\cap A$ and $B'\subset X'\cap B$.  Let
$i:(X';A',B')\to (X;A,B)$ be the inclusion map of triads.
Let $\alpha\in I_{\bar q}H^k(X,B;F)$ and $x\in I^{\bar
r}H_j(X',A'\cup B';F)$. Then 
\[
 \alpha\smallfrown i_*x = 
i_*((i^*\alpha)\smallfrown x)
\]
in $I^{\bar p}H_{j-k}(X,A;F)$. 
\end{proposition}

\begin{proof}
\begin{align*}
\alpha\smallfrown i_*x 
&=
(1\otimes \alpha){\bar d}(i_*x) \\
&=
(1\otimes \alpha)(i_*\otimes i_*){\bar d}(x) \\
&=
(i_*\otimes i^*\alpha){\bar d}(x) \\
&=
i_*((i^*\alpha)\smallfrown x).
\end{align*}
\end{proof}

\begin{proposition}\label{P: cupcap}
Suppose 
$D\bar r\geq D\bar p+D\bar q+D\bar u$.  Let
$\alpha\in I_{\bar p}H^i(X;F)$, $\beta\in I_{\bar q}H^j(X;F)$, and
$x\in I^{\bar r}H_k(X;F)$. 
Then
$$(\alpha\smallsmile \beta)\smallfrown x =
\alpha\smallfrown(\beta\smallfrown x)$$ 
in $I^{\bar u}H_{k-i-j}(X;F)$.
\end{proposition}

\begin{proof}
If we let $\bar v=D(D\bar p+D\bar q)$ and $\bar w=D(D\bar q+D\bar r)$ then
$\alpha\smallsmile\beta\in I_{\bar v}H^*(X;F)$ and $\beta\smallfrown x\in 
I^{\bar w}H_*(X;F)$, and thus both sides of the equation in the Proposition are 
defined. Now we have
\begin{align*}
(\alpha\smallsmile \beta)\smallfrown x
&=
(1\otimes (\alpha\smallsmile \beta)){\bar d}(x) \\
&=
(1\otimes \alpha\otimes \beta)(1\otimes{\bar d}){\bar d}(x) \\
&=
(1\otimes \alpha\otimes \beta)({\bar d}\otimes 1){\bar d}(x)
\ \text{by Proposition \ref{C: coass}} \\
&=
(1\otimes \alpha){\bar d}((1\otimes \beta){\bar d}(x)) \\
&=
\alpha\smallfrown(\beta\smallfrown x).
\end{align*}
\end{proof}

For our next result, note that (because $D\bar t$ is identically $0$) there is 
a cap product $I_{\bar p}H^i(X,A;F)\otimes I^{\bar p}H_j(X,A;F)\to I^{\bar 
t}H_{j-i}(X;F)$.

\begin{proposition}\label{P: capeval}
Let $A$ be an open subset of $X$. Let $\alpha\in I_{\bar p}H^i(X,A;F)$ and 
$x\in I^{\bar p}H_i(X,A;F)$. Then the image of  $\alpha\smallfrown x$ under the 
augmentation $\varepsilon:I^{\bar t}H_0(X;F)\to F$ is $\alpha(x)$.
\end{proposition}

\begin{proof}
\begin{align*}
\varepsilon(\alpha\smallfrown x)
&=
\varepsilon(1\otimes \alpha){\bar d}(x) \\
&=
\alpha(\varepsilon\otimes 1){\bar d}(x) \\
&=
\alpha(x)\ \text{by the relative version of \ref{P: counital}.}
\end{align*}
\end{proof}

\begin{proposition}\label{L: boundary}
Suppose $D\bar r\geq D\bar p+D\bar q$, and let $i:A\into X$ be the inclusion of
an open subset.
\begin{enumerate}
\item
Let $\alpha\in I_{\bar q}H^k(X;F)$ and $x\in I^{\bar r}H_j(X,A;F)$.  Then
$$\bd(\alpha\smallfrown x)
=(-1)^{|\alpha|}(i^*\alpha)\smallfrown (\bd x)$$ 
in $I^{\bar p}H_{j-k-1}(A;F)$,
where $\bd$ is the connecting homomorphism.
\item
Let $\alpha\in I_{\bar q}H^k(A;F)$ and $x\in I^{\bar r}H_j(X,A;F)$. Then
$$\delta(\alpha)\smallfrown x=-(-1)^{|\alpha|}i_*(\alpha\smallfrown\bd x)$$
in $I^{\bar p}H_{j-k}(X;F)$, where $\bd$ and $\delta$ are the connecting 
homomorphisms.
\end{enumerate}
\end{proposition}

\begin{proof}
We prove part 2; part 1 is similar. 

\begin{align*}
\delta(\alpha)\smallfrown x &= (1\otimes \delta(\alpha)){\bar d}(x) \\
&= -(-1)^{|\alpha|}(1\otimes \alpha)(1\otimes \bd){\bar d}(x) \\
&= -(-1)^{|\alpha|}(1\otimes \alpha)(i_*\otimes 1){\bar d}(\bd x) \quad\text{by Proposition \ref{j2} and the relative version of \ref{C: cocomm}} \\
&=-(-1)^{|\alpha|}i_*(\alpha\smallfrown\bd x)
\end{align*}
\end{proof}

We conclude this subsection with a fact which that be needed at one point in 
Section \ref{j4}.  First observe that if $M$ is a nonsingular manifold with
trivial stratification the cross product induces a map
\[
H_*(M;F)\otimes I^{\bar p}H_*(X;F)
\to
I^{\bar p}H_*(M\times X;F)
\]
for any perversity $\bar p$.
This map is an isomorphism by \cite[Corollary 3.7]{GBF20}, and we define the 
cohomology cross product
\[
\times:H^*(M;F)\otimes I_{\bar p}H^*(X;F)
\to
I_{\bar p}H^*(M\times X;F)
\]
to be the composite
\begin{multline*}
H^*(M;F)\otimes I_{\bar p}H^*(X;F) \cong \Hom_F(H_*(M;F),F)\otimes \Hom_F(I^{\bar p}H_*(X;F),F) \to\\
\Hom_F(H_*(M;F)\otimes I^{\bar p}H_*(X;F),F)\cong\Hom_F(I^{\bar p}H_*(M\times X;F),F)\cong I_{\bar p}H^*(M\times X;F).
\end{multline*}

\begin{remark}
\label{l2}
Note that the second map in this composite, and therefore the entire composite,
is an isomorphism whenever either $H_*(M;F)$ or $I^{\bar p}H_*(X;F)$ is
finitely generated.
\end{remark}

\begin{proposition}\label{j6}
Suppose $D\bar r\geq D\bar p+D\bar q$. Let $\alpha\in H^*(M;F)$, 
$x\in H_*(M;F)$, $\beta\in I_{\bar q}H^*(X;F)$, and 
$y\in I^{\bar r}H_*(X;F)$.  Then
\[
(\alpha\times\beta)\smallfrown(x\times y)=
(-1)^{|\beta||x|}(\alpha\smallfrown x)\times(\beta\smallfrown y)
\]
in $I^{\bar p}H_*(M\times X;F)$.
\end{proposition}

\begin{proof}
This is a straightforward consequence of the definitions and the commutativity
of (the outside of) the following diagram (where the $F$ coefficients are tacit
and $Q$ denotes $Q_{\bar p,\bar q}$).
{\small
\[
\xymatrix{
H_*(M)\otimes I^{\bar r}H_*(X)
\ar[ddd]_\times
\ar[rr]^-{{\bar d}\otimes{\bar d}}
\ar[rd]^-{d\otimes d}
&
&
H_*(M)\otimes H_*(M)\otimes I^{\bar p}H_*(X)\otimes I^{\bar q}H_*(X)
\ar[ld]_-{\times\otimes\times}
\ar[d]_\cong\\
&
H_*(M\times M)\otimes I^QH_*(X\times X)
\ar[d]_\times
&
H_*(M)\otimes I^{\bar p}H_*(X)\otimes H_*(M)\otimes I^{\bar q}H_*(X).
\ar[dd]_{\times\otimes\times}\\
&
I^QH_*(M\times X\times M\times X)
&\\
I^{\bar r}H_*(M\times X)
\ar[rr]^-{\bar d}
\ar[ru]^-d
&
&
I^{\bar p}H_*(M\times X)\otimes I^{\bar q}H_*(M\times X)
\ar[lu]_-\times
}
\]}
This diagram commutes by the definition of $\bar d$ and the naturality,
associativity, and commutativity properties of the cross product.
\end{proof}

\section{Fundamental classes}\label{S: fund}

Recall that the basic homological theory of oriented 
$n$-manifolds has three parts:  the calculation of $H_*(M,M-\{x\})$ and the 
construction of the local orientation class; the construction of the 
fundamental class in $H_*(M,M-K)$ for $K$ compact; and the 
calculation of $H_i(M)$ for $i\geq n$.  In this section we show that all of 
these have analogues for stratified pseudomanifolds using the $0$ perversity:

\begin{definition}
$\bar 0$ is the perversity which is 0 for all strata.
\end{definition}

\begin{remark}
If $X$ is a stratified pseudomanifold with no codimension one strata, then $\bar 0\leq \bar t$ and, for PL pseudomanifolds, $I^{\bar 0}H_*(X)$ agrees with the $\bar 0$ perversity intersection homology groups of Goresky and MacPherson \cite{GM1, GM2}. If $X$ has a non-empty codimension one stratum $S$, then $\bar 0(S)=0>\bar t(S)=-1$, and our definition of $I^{\bar 0}H_*(X)$ would disagree with the direct generalization of the definition of Goresky and MacPherson. See our prior discussion beginning on page \pageref{notation}.
\end{remark}

We begin with an overview of the main results, which will be proved in later
subsections.

Let $R$ be a ring, and let $X$ be an  $n$-dimensional stratified 
pseudomanifold.  As usual, we do not assume that $X$ is compact or connected 
and we allow strata of codimension one. 

Let $X_n$ denote $X-X^{n-1}$.  Recall the following definition from 
\cite[Section 5]{GM2}.

\begin{definition}
An {\em $R$-orientation of $X$} is an $R$-orientation of the manifold $X_n$.  
\end{definition}

Our first goal is to understand $I^{\bar 0}H_*(X,X-\{x\};R)$ (assuming $X$ is
$R$-oriented).

To begin with we note that for $x\in X_n$ we have $I^{\bar 
0}H_*(X,X-\{x\};R)\cong H_*(X_n,X_n-\{x\};R)$ by excision.  In particular the 
usual local orientation class in $H_n(X_n,X_n-\{x\};R)$ determines a local
orientation class $o_x\in I^{\bar 0}H_n(X,X-\{x\};R)$.

Next  we consider the case when $X$ is {\em normal} (that is, when its 
links are connected).\footnote{This differs from the definition of normal
given in \cite[Section 5.6]{GM2} but is equivalent in the cases considered
there.}
The following proposition generalizes a standard result for $R$-oriented 
manifolds.

\begin{proposition}\label{P: local o}
Let $X$ be a normal $R$-oriented $n$-dimensional stratified pseudomanifold. 
\begin{enumerate}
\item For all $x\in X$ and all $i\neq n$, $I^{\bar 0}H_{i}(X,X-\{x\};R)=0$.
\item
The sheaf generated by the presheaf $U\to I^{\bar 0}H_{n}(X,X-\bar U;R)$  is 
constant, and there is a unique global section $s$ whose value at each $x\in 
X_n$ is $o_x$.
\item
For all $x\in X$,  $I^{\bar 0}H_{n}(X,X-\{x\};R)$ is the free $R$-module
generated by $s(x)$.
\end{enumerate}
\end{proposition}

\begin{definition} \label{D: normal local class}
Let $X$ be a normal $R$-oriented stratified pseudomanifold, and let $x\in X$.
Define the local orientation class $o_x\in I^{\bar 0}H_n(X,X-\{x\};R)$
to be $s(x)$.
\end{definition}

Now we recall that Padilla \cite{Pa03} constructs a
{\em normalization} $\pi:\hat X\to X$ for each stratified pseudomanifold $X$.
Here $\hat X$ is normal and $\pi$ has the properties given in \cite[Definition
2.2]{Pa03}; in particular $\pi$ is a finite-to-one map which
induces a homeomorphism from $\hat X-\hat X^{n-1}$ to $X-X^{n-1}$. Padilla
shows that the normalization is unique up to isomorphism (that is, up to
isomorphism there is a unique $\hat X$ and $\pi$ satisfying \cite[Definition 
2.2]{Pa03}).

\begin{proposition}\label{P: abnormal local o}
Let $X$ be an $R$-oriented $n$-dimensional stratified pseudomanifold, not
necessarily normal.  Give $\hat X$ the $R$-orientation induced by $\pi$.
Let $x\in X$.
\begin{enumerate}
\item For all $i\neq n$, $I^{\bar 0}H_{i}(X,X-\{x\};R)=0$.
\item
$I^{\bar 0}H_{n}(X,X-\{x\};R)$ is the free $R$-module
generated by the set $\{\pi_*(o_y)\,|\, y\in \pi^{-1}(x)\}$.
\end{enumerate}
\end{proposition}

\begin{definition}\label{D: non-normal local class}
Let $X$ be an $R$-oriented stratified pseudomanifold and give $\hat X$ the
$R$-orientation induced by $\pi$.  For $x\in X$, define the local orientation
class $o_x\in I^{\bar 0}H_n(X,X-\{x\};R)$ to be
\[
\sum_{y\in \pi^{-1}(x)}\, \pi_*(o_y).
\]
\end{definition}

This is consistent with Definition \ref{D: normal local class} because for 
normal $X$, $\pi$ is the identity map.

Our next result constructs the fundamental class.

\begin{theorem}\label{T: global o}
Let $X$ be an $R$-oriented stratified pseudomanifold.  For each compact
$K\subset X$, there is a unique $\Gamma_K\in I^{\bar 0}H_n(X,X-K;R)$ that restricts to $o_x$ for each $x\in K$.
\end{theorem}

\begin{definition}
Define the {\em fundamental class of $X$ over $K$} to be $\Gamma_K$.
\end{definition}

\begin{remark}\label{j7}
For later use we note that if $K\subset K'$ then the map $I^{\bar 0}H_n(X,X-K';R)\to I^{\bar 0}H_n(X,X-K;R)$ takes $\Gamma_{K'}$ to $\Gamma_K$.
\end{remark}

We then verify that  $I^{\bar 0}H_i(X;R)=0$ for $i\geq n$ when $X$ is 
compact and describe how the intersection homology of $X$ can be decomposed into summands corresponding to the regular strata of $X$.  Note that if $Z$ is a regular stratum of $X$ then the closure 
$\bar{Z}$ (with the induced filtration) is a stratified pseudomanifold; this 
follows from a straightforward induction over the depth of $X$.

\begin{theorem}\label{x1}
Let $X$ be a compact $R$-oriented $n$-dimensional stratified pseudomanifold.

\begin{enumerate}
\item
$I^{\bar 0}H_i(X;R)=0$ for $i>n$.

\item
The natural map
\[
\bigoplus_Z\, I^{\bar 0}H_n(\bar{Z};R)\to I^{\bar 0}H_n(X;R)
\]
is an isomorphism,
where the sum is taken over the regular strata of $X$.

\item If $Z$ is a regular stratum of $X$, then 
$I^{\bar 0}H_n(\bar{Z};R)$ is the free $R$-module generated by
the fundamental class of $\bar{Z}$.

\end{enumerate}
\end{theorem}

\begin{remark}
If $X$ is connected and normal then $X$ has only one regular stratum, by
\cite[Lemma 2.1]{Pa03}.  So in this case we have $I^{\bar 0}H_n(X;R)\cong R$.
(This also follows from the second proposition in \cite[Section 5.6]{GM2}, but
that result does not give the relation with the local orientation classes.)
\end{remark}

Here is an outline of the rest of the section.  We prove Proposition 
\ref{P: local o} and Theorems \ref{T: global o} and \ref{x1} (assuming $X$ is
normal) in Subsections \ref{S: local o} and \ref{S: global o}.  We deduce 
Proposition \ref{P: abnormal local o} and the general case of Theorems
\ref{T: global o} and \ref{x1} in Subsection  \ref{x2}.  In Subsection 
\ref{x3}, we give some further properties of the fundamental class, and in 
Subsection \ref{S: fund indep}, we show that if $X$
is compact and has no strata of codimension one then $\Gamma_X$ is 
independent of the stratification. 

\begin{remark}In this section we focus attention on the perversity 
$\bar 0$ because this is where the fundamental class needed for our duality
results lives, but much of our work is also valid for other perversities and
even for ordinary homology:
\begin{enumerate}
\item
For any nonnegative perversity $\bar p$,
Propositions \ref{P: local o} and \ref{P: abnormal local o} and Theorems
\ref{T: global o} and \ref{x1} all hold with $I^{\bar 0}H_*$ replaced by 
$I^{\bar p}H_*$, except that it is not true in general that 
$I^{\bar p}H_*(X,X-\{x\};R)=0$ for $*<n$.  The proofs are exactly the same.

\item
If $X$ is normal and has no codimension one strata then Proposition \ref{P:
local o} and Theorems \ref{T: global o} and \ref{x1} hold with $I^{\bar 0}H_*$
replaced by ordinary homology $H_*$, except that it is not true that
$H_*(X,X-\{x\};R)=0$ for $*<n$.  Again, the proofs are exactly the same.

\end{enumerate}
\end{remark}

\subsection{The orientation sheaf}\label{S: local o}

Our goal in this subsection will be to prove Proposition \ref{P: local o}, while in the next subsection we prove Theorems \ref{T: global o} and \ref{x1} under the assumption that $X$ is normal. 
The general plan of the proofs is the same as in the classical case when $X$ is a manifold; see e.g. \cite[Theorem 3.26]{Ha} for a recent reference. However, there are some technical issues that need to be overcome due to the lack of homogeneity of $X$. 

The proofs of the proposition and the theorems proceed by a simultaneous 
induction on the depth of $X$. Note that if the depth of $X$ is $0$, then $X$ is a manifold, and all results follow from the classical manifold theory.
In the remainder of this subsection, we prove Proposition \ref{P: local o} under the assumption that Proposition \ref{P: local o} and Theorems \ref{T: global o} and \ref{x1} have been proven for normal stratified pseudomanifolds of depth less than that of $X$. In the next section, we will then use Proposition \ref{P: local o} to prove Theorems \ref{T: global o} and \ref{x1} at the depth of $X$.

\begin{proof}[Proof of Proposition \ref{P: local o}]
Recall we let $X_i=X^i-X^{i-1}$.  

We first show that for any $x\in X$, $I^{\bar 0}H_{n}(X,X-\{x\};R)\cong R$ and
$I^{\bar 0}H_{i}(X,X-\{x\};R)=0$ for $i\neq n$. This is trivial for $x\in
X_n$. 
For 
$x\in X_{n-k}$, we may assume that $x$ has a distinguished
neighborhood of the form $N\cong \R^{n-k}\times cL^{k-1}$. By excision (see
\cite{GBF10}), $I^{\bar 0}H_{i}(X,X-\{x\};R)\cong I^{\bar 0}H_{i}(N,N-\{x\};R)$, and by
the K\"unneth theorem \cite{Ki} with the unfiltered $(\R^{n-k},\R^{n-k}-0)$, this is
isomorphic to $I^{\bar 0}H_{i-(n-k)}(cL, cL-\{x\};R)$. By the cone formula, this is
$I^{\bar 0}H_{i-(n-k)-1}(L; R)$ for $i-(n-k)> k-1$ (i.e. for $i\geq n$) and $0$
otherwise.  The link $L$ is compact by the definition of a stratified pseudomanifold, it is connected since $X$ is normal, and it has depth less than that of $X$. So by induction, $I^{\bar 0}H_i(L;R)=0$ for $i>k-1$ and,  given an $R$-orientation of $L$ (which we shall find in a moment), we have  $I^{\bar 0}H_{k-1}(L;R)\cong R$ with a preferred generator $\Gamma_L$ representing the local orientation class. It follows that $I^{\bar 0} H_{i}(X,X-\{x\};R)=0$ for $i\neq n$ and $I^{\bar 0}H_{n}(X,X-\{x\};R)\cong R$ for any $x\in X$. 

The $I^{\bar 0}H_{n}(X,X-\{x\};R)$ are the stalks of the sheaf  $\mc O^X$ generated by the  presheaf $U\to I^{\bar 0} H_{n}(X,X-\bar U;R)$.  
We next   show this is a locally constant sheaf.  Certainly it is a locally-constant sheaf over $X_n$ by manifold theory. So assume by induction hypothesis that this sheaf is locally-constant over $X-X^{n-k}$. Let $x\in X_{n-k}$, and  again choose a distinguished neighborhood $N\cong \R^{n-k}\times cL$.  To appropriately orient $L$, we assume that $L$ is embedded in $cL$ as some $b\times L$ with $b\in (0,1)$, and
we use that for any choice $z\in N\cap X_n$, there is a local orientation class $o_z\in I^{\bar 0} H_n(X,X-\{z\};R)$, determined by the orientation of $X$. In particular, let $z\in L\cap X_n\subset N$, which we can write as  $z=(0,b,c)$ for $0\in \R^{n-k}$, $b\in (0,1)$ (along the cone line), and $c\in L-L^{k-2}$. Then $I^{\bar 0}H_n(X,X-\{z\};R)\cong H_{n-k}(\R^{n-k},\R^{n-k}-\{0\};\Z) \otimes H_{1}((0,1),(0,1)-\{b\};\Z)\otimes I^{\bar 0}H_{k-1}(L,L-\{c\};R)$. Choosing the canonical generators of $H_{n-k}(\R^{n-k},\R^{n-k}-\{0\};\Z)$ and $H_{1}((0,1),(0,1)-\{b\};\Z)$, the local orientation class of $I^{\bar 0}H_n(X,X-\{z\};R)$ thus determines a local orientation class of $I^{\bar 0}H_{k-1}(L,L-\{c\};R)$. Since $c\in L-L^{k-2}$ was arbitrary but the generators of $H_{n-k}(\R^{n-k},\R^{n-k}-\{0\};\Z)$ and  $ H_{1}((0,1),(0,1)-\{b\};\Z)$ are fixed and we know the generator of $I^{\bar 0}H_n(X,X-\{z\};R)$ remains constant over $L\cap X_n$, this determines a fixed $R$-orientation of $L$ and hence a choice of $\Gamma_L$.
 
Now, having chosen the $R$-orientation for $L$ and a corresponding fundamental class $\Gamma_L$,  a more careful look at the usual K\"unneth and cone formula arguments show that  $I^{\bar 0}H_{n}(N,N-\{x\};R)\cong R$ is generated by $[\eta]\times \bar c\Gamma_L$, where $\eta$ is a chain representing the local orientation class of $H_{n-k}(\R^{n-k}, \R^{n-k}-0;\Z)$ and $\bar c\Gamma_L$ is the singular chain cone on $\Gamma_L$. More explicitly, if we let  $\xi$ stand for a specific intersection chain representing the class $\Gamma_L$, continuing to consider $L$ as the subset $\{b\}\times L\subset cL$,  then $\bar c\Gamma_L$ is represented by the chain $\bar c\xi$ formed by extending each simplex $\sigma$ in $\xi$ to the singular cone simplex $[v,\sigma]$, where $v$ represents the cone point of $cL$. 

Now, if we assume $x$  lives at $0\times v\in N\cong \R^{n-k}\times cL$,
where $0$ is the origin in $\R^{n-k}$ and $v$ is the cone point of $cL$, then we can take a smaller neighborhood  $N'$ of $x$ with $N'\cong B_{\delta}\times \bar c_{\epsilon}L$, where $B_{\delta}$ is the ball of radius $\delta$ about the origin in $\R^{n-k}$ and $\bar c_{\epsilon}L$ is $([0,\epsilon]\times L)/\sim$ within the cone $([0,1)\times L)/\sim$ (where $\sim$ collapses $L\times 0$ to a point). We choose $\delta$ so that the image of $\eta$ in $H_{n-k}(\R^{n-k},\R^{n-k}-\{a\};\Z)$ is a generator for all $a\in B_{\delta}$, and we let $\epsilon<b$, where again we have embedded $L$ in $cL$ at distance $b$ from the vertex. 
With these choices,  the chain $[\eta]\times \bar c\Gamma_L$ not only generates $I^{\bar 0}H_{n}(N,N-\{x\};R)\cong I^{\bar 0}H_{n}(X,X-\{x\};R)\cong R$, but it also   restricts to the local orientation class $o_z\in I^{\bar 0} H_n(X,X-\{z\};R)$ for any  $z\in N'\cap X_n$.
 
 Next we observe that
$I^{\bar 0}H_n(X,X-N';R)\cong I^{\bar 0}H_n(X,X-\{x\};R)$ via the inclusion map since $X-N'$ is stratified homotopy equivalent to $X-\{x\}$ (see Appendix \ref{A: sphe}), and furthermore, $[\eta]\times \bar c\Gamma_L$ generates both groups. Similarly, it generates $I^{\bar 0}H_n(X,X-\{x'\};R)$ for any other $x'\in N'\cap X_{n-k}$. This is enough to guarantee that our orientation sheaf is locally constant along $N'\cap X_{n-k}$. But now also if $z$ is any point in the top stratum of $N'$, we have seen that  $o_z\in I^{\bar 0}H_n(X,X-\{z\};R)\cong R$ is also represented by $[\eta]\times\bar c \Gamma_L$ (and in a way that preserves the choice of orientation). But then by the induction hypotheses, this chain must also restrict to  $o_{z'}\in I^{\bar 0}H_n(X,X-\{z'\};R)\cong R$ for any $z'\in N'-N'\cap X_{n-k}$. So $\mc O^X$ is constant on $N'$, and  for $x'\in N'\cap X_{n-k}$, we can now  let $o_{x'}\in  I^{\bar 0}H_n(X,X-\{x'\};R)$ be the image of $[\eta]\cap \bar c\Gamma_L$. 

It now follows by induction that $\mc O^X$ must be locally constant. Furthermore, over sufficiently small distinguished neighborhoods $N'$,  we have found local sections that restrict to $o_x$ for each $x\in N'\cap X_n$. If $U,V$ are any two such open sets of $X$ with corresponding local sections $s_U,s_V$, then we see that $s_U$ and $s_V$ must therefore agree on $U\cap V\cap X_n$. But it follows from the local constancy that they must therefore also agree on all of $U\cap V$. Therefore, we can piece together the local sections that agree with the local orientation classes into a global section, and it follows that $\mc O^X$ is constant. 

\end{proof}

\subsection{Fundamental classes and global computations}\label{S: global o}

In this section, we provide the combined proofs of Theorems \ref{T: global o} 
and \ref{x1} assuming $X$ is normal and assuming Proposition \ref{P: local o} 
up through the depth of $X$. In order to prove Theorem \ref{x1} we need a
somewhat stronger version of Theorem \ref{T: global o}: 

\begin{proposition}\label{L: fundamental}
Suppose $X$ is a normal $n$-dimensional  $R$-oriented stratified 
pseudomanifold.
Then for any compact $K\subset X$, $I^{\bar 0}H_i(X, X-K;R)=0$ for $i>n$, and 
there is a unique class $\Gamma_K\in I^{\bar 0}H_n(X,X-K;R)$ such that for 
any $x\in K$, the image of $\Gamma_K$ in $X$ is the local orientation class   
$o_x\in I^{\bar 0}H_n(X,X-\{x\};R)$. Furthermore, if $\eta\in I^{\bar 0}H_n(X,X-K;R)$ is such that the image of $\eta$ in $I^{\bar 0}H_n(X,X-\{x\};R)$ is zero for all $x\in K$, then $\eta=0$. 
\end{proposition}

\begin{proof}[Proof of Proposition \ref{L: fundamental}]
We first observe that if the proposition is true for compact sets $K$, $K'$, and $K\cap K'$, then it is true for $K\cup K'$. This follows from a straightforward Mayer-Vietoris argument exactly as it does for manifolds (see \cite[Lemma 3.27]{Ha}). Also analogously to the manifold case in \cite{Ha}, we can reduce to the situation where $X$ has the form of a distinguished neighborhood $X\cong \R^{n-k}\times cL^{k-1}$. To see this, we note that any compact $K\subset X$ can be written as the union of finitely many compact sets $K=K_1\cup \ldots \cup K_m$ with each $K_i$ contained in such a distinguished neighborhood in $X$. This can be seen by a covering argument using the compactness of $K$. Now, notice that $(K_1\cup\cdots \cup K_{m-1})\cap K_m=(K_1\cap K_m)\cup\cdots\cup(K_{m-1}\cap K_m)$ is also a union of $m-1$ compact sets each contained in a distinguished neighborhood, so the Mayer-Vietoris argument and an induction on $m$ reduces matters to the base case of a single $K$ inside a distinguished neighborhood. Then by excision, we can assume that $X=\R^{n-k}\times cL^{k-1}$. 

The next step in the classical manifold setting, as discussed in \cite{Ha}, would be to consider the case where $K$ is a finite union of convex sets in $\R^n$. This  is not available to us in an obvious way. However, let us define a compact set in $X=\R^{n-k}\times cL^{k-1}$ to be PM-convex (``PM'' for pseudomanifold) if either 

\begin{enumerate}
\item it has the form $C\times ( [0,b]\times L)/\sim$, where $C$ is a convex set in $\R^{n-k}$, and $( [0,b]\times L)/\sim$ is part of the cone on $L$, including the vertex, or

\item it has the form $C\times [a,b] \times A $, where $C$ is a convex set in $\R^{n-k}$, $A$ is a compact subset of $L$, and $[a,b]$ is an interval along the cone line with $a>0$.

\end{enumerate}

It is clear that the intersection of any two PM-convex sets is also a PM-convex set, so another Mayer-Vietoris argument and induction allows us to reduce to the case of a single PM-convex set.  For PM-convex sets of the second type, we can use an excision argument to cut $\R^{n-k}\times \{v\}$ (where $v$ is the cone vertex) out of $I^{\bar 0}H_*(X,X-K;R)$ and then appeal to an induction on depth. 
For a PM-convex set $K$ of the first type, computations exactly as in the proof of Proposition \ref{P: local o} show that there is a class $\Gamma_K$ with the desired restrictions to $I^{\bar 0}H_n(X,X-\{x\};R)$ for each $x\in K$. Now, by stratified homotopy equivalence (see Appendix \ref{A: sphe}), $I^{\bar 0}H_*(X,X-K;R)\cong I^{\bar 0} H_*(X,X-\{x\};R)$ for any $x\in C\times v$, and by the usual computations then $I^{\bar 0}H_*(X,X-K;R)\cong I^{\bar 0}H_*(X,X-\{x\};R)\cong I^{\bar 0}H_{*-{n-k}-1}(L;R)$, which is $0$ for $*>n$ and $R$ for $*=n$. Therefore, any other element of $I^{\bar 0}H_n(X,X-K;R)$ that is not $\Gamma_K$ cannot yield the correct generator of  $I^{\bar 0}H_n(X,X-\{x\};R)$ upon restriction, and so $\Gamma_K$ is unique. Furthermore, we see that if $\eta\in I^{\bar 0}H_n(X,X-K;R)$ restricts to $0$ in $I^{\bar 0}H_n(X,X-\{x\};R)$ for any $x\in C\times v\subset K$ then $\eta=0$ (and so certainly $\eta=0$ if $\eta$ goes to $0$ in $I^{\bar 0}H_n(X,X-\{x\};R)$ for every $x\in K$).

Next, we consider an arbitrary compact $K$ in $\R^{n-k}\times cL^{k-1}$ and again follow the general idea from \cite{Ha}. For the existence of a $\Gamma_K$, let $\Gamma_K$ be the image in $I^{\bar 0}H_n(X,X-K;R)$ of any $\Gamma_D$, where $D$ is any PM-convex set sufficiently large to contain $K$. It is clear that such a $D$ exists using that our space has the form $\R^{n-k}\times cL^{k-1}$. By applying the results of the preceding paragraph for the PM-convex case, we see that $\Gamma_K$ has the desired properties. To show that $\Gamma_K$ is unique, suppose that $\Gamma'_K\in I^{\bar 0}H_n(X,X-K;R)$ is another class with the desired properties. Suppose $z$ is a relative cycle representing $\Gamma_K-\Gamma_{K}'\in I^{\bar 0}H_n(X,X-K;R)$. Let $|\bd z|$ be the support of $\bd z$, which lies in $X-K$. Since $|\bd z|$ is also compact, we can cover $K$ by a finite number of sufficiently small PM-convex sets that do not intersect $|\bd z|$. Let $P$ denote the union of these PM-convex sets. The relative cycle $z$ defines an element $\alpha\in I^{\bar 0}H_n(X,X-P;R)$ that maps by inclusion to $\Gamma_K-\Gamma'_K\in I^{\bar 0}H_i(X,X-K;R)$. So $\alpha$ is $0$ in $I^{\bar 0}H_n(X,X-\{x\};R)$ for all $x\in K$. This implies that  the image of $\alpha$ is also $0$ in $I^{\bar 0}H_n(X,X-\{x\};R)$ for all $x$ in $P$. To see this, note that any $x\in P$ is in the same PM-convex set, say $Q$, as some $y\in K$. But then by the preceding paragraph, $I^{\bar 0}H_n(X,X-\{x\};R)\cong  I^{\bar 0}H_n(X,X-Q;R)\cong I^{\bar 0}H_n(X,X-\{y\};R)$.  But now  we must have $\alpha=0$ in $I^{\bar 0}H_n(X,X-P;R)$ since $P$ is a finite union of PM-convex sets. Hence $\Gamma_K-\Gamma'_K=0$. This implies the uniqueness of the class $\Gamma_K$. The same arguments show that if $\eta\in I^{\bar 0}H_n(X,X-K;R)$ goes to $0$ in $I^{\bar 0}H_n(X,X-\{x\};R)$ for each $x\in K$, then $\eta=0$. 

Finally, to see that $I^{\bar 0}H_i(X,X-K;R)=0$ for $i>n$, again let 
$z$ be a relative cycle representing an element $\xi\in I^{\bar 0}H_i(X,X-K;R)$, $i>n$. We can form $P$ exactly as in the preceding paragraph. Once again, the relative cycle $z$ defines an element $\alpha\in I^{\bar 0} H_i(X,X-P;R)$ that maps by inclusion to $\xi\in I^{\bar 0}H_i(X,X-K;R)$. But now if $i>n$, then by the preceding results, $\alpha=0$ since $P$ is PM-convex. Thus also $\xi=0$. 
\end{proof}

For $X$ normal, we can now complete the proof of Theorems \ref{T: global o} and
\ref{x1}.  Theorem \ref{T: global o} follows directly from Proposition 
\ref{L: fundamental}, as well as the first part of  Theorem \ref{x1}  by 
taking $K=X$ if $X$ is compact.

We need to see that if $X$ is compact and connected then $I^{\bar 0} H_n(X;R)\cong R$, generated by $\Gamma_X$. For any $x\in X$, we have a
homomorphism $I^{\bar 0}H_n(X;R)\to I^{\bar 0}H_n(X, X-\{x\};R)\cong R$, which
we know is surjective, sending $\Gamma_X$ onto a local orientation class, by
Proposition \ref{L: fundamental}.  On the other hand, suppose that $\eta\in
I^{\bar 0}H_n(X;R)$ goes to $0\in I^{\bar 0}H_n(X, X-\{x\};R)$ for some $x\in
X$. Since $x\to \im(\eta)\in I^{\bar 0}H_n(X,X-\{x\};R)$ is a section of the
locally-constant orientation sheaf, this implies that $\im(\eta)=0\in I^{\bar 0} H_n(X,X-\{x\};R)$ for all $x\in X$ (since $X$ is connected). But then
$\eta=0$ by Proposition \ref{L: fundamental}.
Thus for any $x\in X$, the homomorphism $I^{\bar 0}H_n(X;R)\to I^{\bar 0}H_n(X, X-\{x\};R)\cong R$ is an isomorphism. If $X$ has multiple compact normal connected components, the rest of the theorem follows by noting that $I^{\bar 0}H_n(X;R)$ is the direct sum over the connected components and by piecing together the results for the individual components.\hfill\qedsymbol

\subsection{\ref{P: abnormal local o}, \ref{T: global o} and \ref{x1}
for general pseudomanifolds}\label{x2}

\begin{proof}[Proof of Proposition \ref{P: abnormal local o}] If $X$ is not necessarily normal, let $\pi:\hat X\to X$ be its normalization. By Lemma \ref{L: normal} of the Appendix \ref{A: normal} (which extends  results well-known for Goresky-MacPherson perversities), $\pi$ induces isomorphisms $I^{\bar 0}H_*(\hat X,\hat X-\pi^{-1}(\bar U);R)\to I^{\bar 0}H_*( X,X-\bar U;R)$. Proposition
\ref{P: local o} then  implies that $I^{\bar 0}H_i( X,X-\{x\};R)=0$ for $i\neq n$ and that $\mc O^X\cong \pi_*\mc O^{\hat X}$. We thus obtain our desired global section of $\mc O^X$  from the preferred global section of $\mc O^{\hat X}$ using the general sheaf theory fact  $\Gamma(X;\pi_*\mc O^{\hat X})=\Gamma(\hat X;\mc O^{\hat X})$. The formula for $I^{\bar 0}H_n( X,X-\{x\};R)$ also follows from basic sheaf theory. 
\end{proof}

Finally, we prove Theorems \ref{T: global o} and \ref{x1} for $X$ not necessarily normal.

\begin{proof}
If $X$ is not necessarily normal, once again there is a map  $\pi:\hat X\to
X$ such that $\hat X$ is normal, $\pi$ restricts to a homeomorphism from $\hat
X-\hat X^{n-1}$ to $X_n=X-X^{n-1}$, and $\pi$ induces an isomorphism on
intersection homology by Lemma \ref{L: normal}. In addition, the number of
connected components of $\hat X$ is equal\footnote{Since $\hat X-\hat
X^{n-1}\cong X-X^{n-1}$ and $\hat X-\hat X^{n-1}$ is dense in $\hat X$, the
number of connected components of $\hat X$ is less than or equal to the number
of components of $X-X^{n-1}$. But each connected normal stratified pseudomanifold has only one regular stratum \cite[Lemma 2.1]{Pa03}, so this must in fact be an equality.} to the number of connected components of $X-X^{n-1}$. Notice that if $Z$ is such a component of $X-X^{n-1}$, $\pi$ restricts to a normalization map from the closure of $\hat Z:=\pi^{-1}(Z)$ in $\hat X$ to the closure of $Z$ in $X$, which is also a stratified pseudomanifold (as follows from a local argument via an induction on depth). Also, if $K$ is a compact subset of $X$, then $\hat K=\pi^{-1}(K)$ is compact since $\pi$ is proper, and $\pi$ restricts to the normalization $\hat X-\hat K\to X-K$ (see Proposition 2.5 and Theorem 2.6 of \cite{Pa03}). 

So $\pi$ also induces isomorphisms $I^{\bar 0}H_*(\hat X,\hat X-\hat K;R)\to I^{\bar 0}H_*(X,X-K;R)$, and since $\hat X$ is normal, our preceding results yield a unique fundamental class $\hat \Gamma_K\in I^{\bar 0}H_n(\hat X,\hat X-\hat K;R)$ and show that $I^{\bar 0}H_i(\hat X,\hat X-\hat K;R)=0$  for $i>n$. But the isomorphism $\pi$ on intersection homology then shows that $I^{\bar 0}H_i( X, X- K;R)=0$ for $i>n$ and provides a  class $\pi\hat \Gamma_K\in I^{\bar 0}H_n( X, X- K;R)$. Let us see that $\pi\hat \Gamma_K$ has the desired properties for it to be $\Gamma_K$. 

By taking sufficiently small distinguished neighborhoods around $x\in X$, letting $\pi^{-1}(x)=\{y_1,\ldots, y_m\}$, and excising, we have the following commutative diagram (coefficients tacit):
{\scriptsize\begin{diagram}[LaTeXeqno]\label{D: restrict}
I^{\bar 0}H_n(\hat X,\hat X-\hat K)&\rTo&I^{\bar 0}H_n(\hat X,\hat X-\pi^{-1}(x))&\lTo^\cong &I^{\bar 0}H_n(\hat N, \hat N-\pi^{-1}(x))\cong \oplus_{i=1}^m I^{\bar 0}H_n(\hat N_i,  \hat N_i-y_i)&\rTo^\cong& \oplus_{i=1}^m R\\
\dTo^\pi_\cong&&\dTo^\pi_\cong&&\dTo^\pi_\cong&&\dTo_\cong\\
I^{\bar 0}H_n(X,X-K)&\rTo&I^{\bar 0}H_n(X,X-x)&\lTo^\cong &I^{\bar 0}H_n(N,  N-x)&\rTo^\cong &\oplus_{i=1}^m R.\\
\end{diagram}}
So by the definition of the local orientation class $o_x$, we see that the image of $\pi\hat \Gamma_K$ in $I^{\bar 0}H_n( X, X- \{x\};R)$ is precisely $o_x$ since $\hat \Gamma_K$ restricts to the local orientation class in each 
$I^{\bar 0}H_n( X, X- \{y_i\};R)$. Restricting this same argument to $\bar Z$ and $cl(\hat Z)$ demonstrates that the image of $\Gamma_{cl(\hat Z)}$ must be $\Gamma_{\bar Z}$.

To see that $\Gamma_K$ is unique, let $\Gamma_K'$ be another class with the
desired properties. Then $\Gamma_K$ and $\Gamma_{K'}$ each correspond to unique
elements $\hat \Gamma_K$ and $\hat \Gamma_{K'}$ in $I^{\bar 0}H_n(\hat X,\hat
X-\hat K;R)$. But now using diagram \eqref{D: restrict} again, we see that
$\hat \Gamma_K$ and $\hat \Gamma_K'$ must each restrict to the local orientation class of $I^{\bar 0}H_n(\hat X,\hat X-\{y\};R)$ for each $y\in \hat K$ or else
their images in  $I^{\bar 0}H_n(X,X-\{x\};R)$ will not be correct. But by the
uniqueness for normal stratified pseudomanifolds, which we obtained in
Proposition \ref{L: fundamental}, this implies $\hat \Gamma_K=\hat \Gamma_K'$, and so $\Gamma_K=\Gamma_K'$.

This completes the proof of Theorem \ref{T: global o}.

If $X_n$ is connected and we 
take $K=X$, then  Theorem \ref{x1} follows immediately from the isomorphism 
$I^{\bar 0}H_*(\hat X; R)\cong I^{\bar 0}H_*(X; R)$.  

If $X_n$ is not connected, let $\bar Z$ be the closure of $Z$ in $X$, let $\text{cl}(\hat Z)$ be the closure of $\hat Z$ in $\hat X$, and notice $\hat X$ is the disjoint union of the connected components $\hat X=\amalg \text{cl}(\hat Z)$. So then

\begin{diagram}[LaTeXeqno]\label{D: pieces}
\oplus_Z I^{\bar 0}H_n(\text{cl}(\hat Z);R) &\rTo^{\cong} & I^{\bar 0}H_n(\hat X;R)\\
\dTo^\pi_{\cong}&&\dTo^\pi_{\cong}\\
\oplus_Z I^{\bar 0}H_n(\bar Z;R) &\rTo^{\cong} & I^{\bar 0}H_n(\hat X;R).
\end{diagram}
The top map is an isomorphism because the spaces $\text{cl}(\hat Z)$ are disjoint. The vertical maps are normalization isomorphisms, and it follows that the bottom is an isomorphism. Thus it follows from the normal case that $I^{\bar 0}H_i(\hat X;R)=0$ for $i>n$ and that $I^{\bar 0}H_n(\hat X;R)\cong R^m$, where $m$ is the number of connected components of $X-X^{n-1}$. 

 The remainder of the Theorem \ref{x1} follows since our earlier arguments imply that the image of $\Gamma_{\text{cl}(\hat Z)}\in I^{\bar 0}H_n(\text{cl}(\hat Z);R)$ under the normalization 
 $\pi|_{\text{cl}(\hat Z)}:\text{cl}(\hat Z)\to \bar Z$ is $\Gamma_{\bar Z}\in I^{\bar 0}H_n(\bar Z;R)$.  
\end{proof}

\subsection{Corollaries and complements}\label{x3}

For later use we record some further properties of the fundamental class.

\begin{corollary}\label{T: uniqueness of fund. class}
Suppose $X$ is a compact connected normal $n$-dimensional $R$-oriented 
stratified pseudomanifold. If $\gamma\in I^{\bar 0}H_n(X;R)\cong R$ is a 
generator, then $\gamma$ is the fundamental class of $X$ with respect to some 
orientation of $X$. 
\end{corollary}
\begin{proof}
If $\Gamma_X$ is the fundamental class of $X$ with respect to the given orientation, then clearly $\gamma=r\Gamma_X$ for some unit $r\in R$. Thus the image of $\gamma$ in any $I^{\bar 0}H_n(X,X-\{x\};R)$ is $r$ times the image of $\Gamma_X$ in $I^{\bar 0}H_n(X,X-\{x\};R)$. Thus $\gamma$ is the fundamental class corresponding to the global orientation section obtained from that corresponding to $\Gamma_X$ by multiplication by $r$. 
\end{proof}

\begin{corollary}\label{T: fund id2}
Suppose $X$ is a compact connected $n$-dimensional $R$-oriented stratified 
pseudomanifold. Let $\{x_i\}$ be a collection of points of $X$, one in each 
connected component of $X-X^{n-1}$. If $\gamma\in I^{\bar 0}H_n(X;R)$ restricts
to $o_{x_i}\in I^{\bar 0}H_n(X, X-\{x_i\};R)$ 
for each $x_i$, then $\gamma=\Gamma_X$. More generally, given 
that $X$ is orientable, any element of $I^{\bar 0}H_n(X;R)$ that restricts to 
a generator of each $I^{\bar 0}H_n(X, X-\{x_i\};R)$ determines an orientation 
of $X$. 
\end{corollary}
\begin{proof}
First assume $X$ is given an orientation and hence has a fundamental class $\Gamma_X$. We know $\Gamma_X$ has the desired property. It is clear from Theorem \ref{T: global o} and diagram \eqref{D: pieces}, that no other element of  $I^{\bar 0}H_n(X;R)$ can have this property, since, as in the proof of the preceding corollary, any other element of   $I^{\bar 0}H_n(X;R)$ would have to restrict to a different element of $I^{\bar 0}H_n(X, X-\{x_i\};R)$ for at least one of the $x_i$. 

Conversely,  an element of $I^{\bar 0}H_n(X;R)$ that restricts to a generator of each $I^{\bar 0}H_n(X, X-\{x_i\};R)$ determines a local orientation at each $x_i$. But since $X$ is orientable, any local orientation at one point of each  regular stratum determines an orientation of $X$. 
\end{proof}

Our next result will be needed in \cite{GBF31}.  It utilizes the definition of
stratified homotopy equivalence given in Appendix \ref{A: sphe}.

\begin{corollary}
Suppose $X$ and $Y$  are  compact  $n$-dimensional  stratified pseudomanifolds, that $X$ is $R$-oriented, and that  $f:X\to Y$ is a stratified  homotopy equivalence. Then $Y$ is orientable and $f$ takes $\Gamma_X$ to $\Gamma_Y$ for some orientation of $Y$. 
\end{corollary}
\begin{proof}
By Padilla \cite[Theorem 2.6]{Pa03}, normalization is functorial, so we have a diagram
\begin{diagram}
I^{\bar 0}H_n(\hat X;R) &\rTo^{\hat f} &I^{\bar 0}H_n(\hat Y;R)\\
\dTo^{\pi_X}&&\dTo_{\pi_Y}\\
I^{\bar 0}H_n( X;R) &\rTo^{f} &I^{\bar 0}H_n(Y;R).
\end{diagram}
The bottom map is an isomorphism by the invariance of intersection homology under stratified homotopy (see Appendix \ref{A: sphe}), and the vertical maps are isomorphisms by
normalization. Hence the top map is also an isomorphism. Borrowing the notation
from above, since each connected component $\text{cl}(\hat Z)$ of $\hat X$ is a
compact connected oriented normal stratified pseudomanifold, with its orientation coming from that of the stratum $Z\subset X$, we have each $I^{\bar 0}H_n(\text{cl}(\hat Z);R)\cong R$, and the fundamental class 
$\Gamma_{\hat X}$ determined by $\Gamma_X$ is a sum of generators of the separate 
$I^{\bar 0}H_n(\text{cl}(\hat Z);R)$. Since $X$ and $Y$ are stratified homotopy equivalent, there is a bijection between connected components of $X-X^{n-1}$ and $Y-Y^{n-1}$, and $f$ induces homotopy equivalences of these manifolds. It follows that $Y$ must be orientable and  that $\hat f(\Gamma_{\hat Y})$ must similarly be  a sum of generators of the corresponding $I^{\bar 0}H_n(\text{cl}(\hat S);R)$ for regular strata $S\subset Y$. By Corollary \ref{T: uniqueness of fund. class}, these generators must be fundamental classes for the $\text{cl}(\hat S)$ with respect to some orientation on $\hat Y$. This determines an orientation of $Y$ by the homeomorphism 
$\hat Y-\hat Y^{n-1}\cong Y-Y^{n-1}$, and, it follows from the diagram that 
$f(\Gamma_X)=\pi_Y\hat f(\Gamma_{\hat X})$ is the corresponding fundamental class on $Y$.
\end{proof}

We next present a result  that will be needed at one point in
Section \ref{j4}.  First observe that if $M$ is an $R$-oriented manifold and 
$X$ is an $R$-oriented
stratified pseudomanifold of dimension $n$ there is 
a canonical $R$-orientation on $M\times X$ (namely the product orientation on 
$M\times X_n$; see \cite[VIII.2.13]{Dold}).

\begin{proposition}\label{j8}
Let $M$ be an $R$-oriented manifold and let $X$ be an $R$-oriented stratified 
pseudomanifold. Let $K_1\subset
M$ and $K_2\subset X$ be compact. Then $\Gamma_{K_1}\times
\Gamma_{K_2}=\Gamma_{K_1\times K_2}$ in $I^{\bar 0}H_*(M\times X,M\times
X-K_1\times K_2)$.
\end{proposition}

\begin{proof}
It suffices to show that for each $x\in M$ and $y\in X$ we have
$o_{(x,y)}=o_x\times o_y$.  

First suppose $X$ is normal. 
Applying Proposition \ref{P: local o}(2) to $M$, $X$ and $M\times X$ gives
sections $s$, $t$ and $u$.  Then $s\times t$ is a continuous section which 
agrees with $u$ for regular points by \cite[VIII.2.13]{Dold} and hence for all
points by the uniqueness property in Proposition \ref{P: local o}(2).
Thus we have
\[
o_{(x,y)}
=u(x,y) 
=s(x)\times t(y) 
=o_x\times o_y
\]
for all $x\in M$ and $y\in X$.

For general $X$, let $\pi:\hat{X}\to X$ be
a normalization.  Then $1\times\pi:M\times \hat{X}\to M\times X$ 
is also a normalization by \cite[Example 2.3(3)]{Pa03}, and if $x\in M$, 
$y\in X$ we have
\begin{align*}
o_x\times o_y
&=
o_x
\times
\Bigl(\sum_{z\in \pi^{-1}(y)} \pi_*(o_z)\Bigr)\\
&=
\sum_{(x,z)\in (1\times \pi)^{-1}(x,y)} (1\times\pi)_*(o_{(x,z)})\\
&=
o_{(x,y)}.
\end{align*}
\end{proof}

\begin{remark}
The analogous fact is true for a product of two pseudomanifolds, but the proof
is more involved (because one has to show that the product of two
normalizations is a normalization).
\end{remark}

\paragraph{A generalization of a result of Goresky and MacPherson.} 
In \cite{GM1}, Goresky and MacPherson consider piecewise-linear (PL) stratified pseudomanifolds without codimension one strata; see \cite{GM1} and \cite[Sections I and II]{Bo} for more on PL pseudomanifolds. If $X^n$ is such a space with an $R$-orientation, then, given a choice of triangulation compatible with the stratification,  one may obtain a fundamental class $\Gamma\in H_n(X;R)$ just as one does for $R$-oriented PL manifolds, represented by a cycle consisting of  a sum over all  $n$-simplices of $X$, properly oriented and thought of as singular simplices via a total ordering of the vertices of $X$. In particular, at each $x\in X-X^{n-1}$, such a $\Gamma$ restricts to the preferred generator of $H_n(X,X-\{x\};R)$, which is the definition given for the fundamental class in \cite[Section 1.4]{GM1}. But also, 
it is easy to check that each simplex of this cycle is $\bar 0$-allowable, and,
by excision and the definitions,  $H_n(X,X-\{x\};R)\cong I^{\bar 0}H_n(X,X-\{x\};R)$ for each $x\in X-X^{n-1}$. So the same cycle represents our fundamental class $\Gamma_X\in I^{\bar 0}H_*(X;R)$, by Corollary \ref{T: fund id2}.

Each perversity considered in \cite{GM1} depends only on codimension and satisfies $\bar p(2)=0$ and 
$\bar p(k)\leq \bar p(k+1)\leq \bar p(k)+1$, where here $k$ and $k+1$ are input codimensions. Continuing to assume $X$ has no codimension one stratum, such perversities satisfy $\bar p\leq \bar t$, and each of the PL intersection homology groups $I^{\bar p}H_*(X)$ of \cite{GM1} is isomorphic to the corresponding groups as defined here (see \cite{GBF35} for details). In the setting of \cite{GM1}, Goresky and MacPherson observe that the cap product in ordinary (co)homology $\smallfrown\Gamma: H^i(X)\to H_{n-i}(X)$ factors as 
$$H^i(X)\xr{\alpha_{\bar p}} I^{\bar p}H_{n-i}(X)\xr{\omega_{\bar p}} H_{n-i}(X).$$ Here the second arrow is induced by the inclusion of PL chain groups\footnote{Given the assumptions on perversities, no allowable PL $i$-simplex can intersect $X^{n-1}=X^{n-2}$ with dimension greater than $i-2$. Thus no simplex of the  boundary $\bd \xi$ of an $i$-chain $\xi$ can lie in $X^{n-2}$, and so the boundary map here is just the ordinary boundary map. Thus this inclusion makes sense as a chain map, even with a PL version of the definition of intersection chains presented in this paper.}  $I^{\bar p}C_*(X)\into C_*(X)$. The first map is described in \cite{GM1} in two different ways, the first utilizing mock bundles to represent elements of $H^i(X)$ as embedded geometric cycles satisfying the $\bar 0$ perversity condition (see \cite{GM1, BRS, Go81}) and the second utilizing a dual cell decomposition of $X$ to describe the cap product with the fundamental class.

The cap product presented here shows that if we use field coefficients $F$, then for an $F$-oriented PL stratified pseudomanifolds without codimension one strata and for perversity $\bar p$ such that $\bar p\leq \bar t$ and  $D\bar p\leq \bar t$, the cap product $\smallfrown\Gamma: H^i(X;F)\to H_{n-i}(X;F)$ factors as a composition
\begin{equation}\label{E: alpha}
H^i(X;F)\xr{\omega_{D\bar p}^*}I_{D\bar p}H^{i}(X;F)\xrightarrow{\smallfrown \Gamma} 
I^{\bar p}H_{n-i}(X;F)\xr{\omega_{\bar p}} H_{n-i}(X;F).
\end{equation}
Here both the first and last maps are induced by inclusions, $I^{D\bar p}C_*(X;F)\into C_*(X;F)$ and 
$I^{\bar p}C_*(X;F)\into C_*(X;F)$, respectively, and the middle map is our intersection (co)homology cap product with the fundamental class in $I^{\bar 0}H_n(X;F)$. Using our observation that, for a PL pseudomanifold without codimension one strata, $\Gamma$ can be represented by a cycle that also represents the Goresky-MacPherson fundamental class in $H_n(X;F)$,  it is not difficult to verify that   the composition corresponds to the standard cap product by the Goresky-MacPherson fundamental class.
It is interesting to notice that the construction of the intersection (co)homology cap product uses different approximations to the diagonal depending on the perversity $\bar p$, but since the cap product in ordinary homology is independent of the choice of diagonal approximation, the composition of these  maps  always corresponds to the ordinary cap product $\smallfrown\Gamma: H^i(X;F)\to H_{n-i}(X;F)$.

In the setting of \emph{topological} compact $F$-oriented stratified pseudomanifolds  we do not have mock bundles or dual cell decompositions available. However,  we can nonetheless  \emph{define} $\bar \alpha_{\bar p}$ to be the composition of the first two maps of \eqref{E: alpha}. Continuing to assume that $X$ has no codimension one strata and that $\bar p\leq \bar t$ and $D\bar p\leq \bar t$, it is  possible to demonstrate  that  $\bar \alpha$, $\omega$, and the intersection product  possesses relations analogous to those  observed by Goresky and MacPherson in \cite[Section 2.4]{GM1}. Details will be provided in \cite{GBF35}; see also \cite{GBF30} for the sheaf-theoretic viewpoint on the relations between the duality discussed in this paper and the Goresky-MacPherson intersection products of \cite{GM1, GM2}. It is likely that $\bar \alpha_{\bar p}$ is identical to the Goresky-MacPherson $\alpha_{\bar p}$ when $X$ is a compact PL stratified pseudomanifold without codimension one strata, though we will not pursue this here.

\subsection{Topological invariance of the fundamental class}
\label{S: fund indep}

In this subsection, we demonstrate a fact needed in \cite{GBF31}: the 
fundamental class of a compact oriented stratified pseudomanifold with no 
codimension one strata is an oriented homeomorphism invariant.

The proscription on codimension one strata is necessary because the group $I^{\bar 0}H_n(X;R)$ itself depends on the stratification of $X$ if codimension one strata are allowed. For example, let $S=S^1$ be the unit circle   stratified trivially, and let $S'$ be the circle stratified as $S^1\supset \{x,y\}$, where $x,y$ are any two distinct points of $S^1$. Then simple computations reveal that $I^{\bar 0}H_1(S;R)\cong H_1(S;R)\cong R$ but 
$I^{\bar 0}H_1(S';R)\cong H_1(S^1,\{x,y\};R)\cong R\oplus R$. 

For the remainder of this subsection, we limit discussion to stratified pseudomanifolds with no codimension one strata. 

We recall from \cite{Ki} that for any stratified pseudomanifold $X$ there is 
an intrinsic coarsest ``stratification'' $X^*$ of $X$ (which is actually a 
CS-set, not a stratified pseudomanifold) that depends only on $X$ as a 
topological space.  The inclusion map 
$I^{\bar 0}H_*(X;R)\to I^{\bar 0}H_*(X^*;R)$ is an isomorphism, and hence if $X'$ is a restratification of 
$X$ (that is, the space $X$ with an alternate pseudomanifold stratification) 
there is a canonical composite isomorphism
\[
I^{\bar 0}H_*(X;R)\xrightarrow{\cong} 
I^{\bar 0}H_*(X^*;R) \xleftarrow{\cong} I^{\bar 0}H_*(X';R).
\]

We first show that an orientation of $X$ determines an orientation on each
restratification of $X$ and on $X^*$.

\begin{lemma}\label{L: orient}
Let $X$ be an oriented stratified pseudomanifold. Let $X'$ be a 
restratification of $X$. Then $X'$ has a unique orientation so that the 
induced orientations on $X_n\cap X'_n$ from $X$ and $X'$ agree. This remains 
true with $X'$ replaced by $X^*$.
\end{lemma} 
\begin{proof}
The proof is the same for $X'$ or $X^*$. Note that since $X^*$ is coarser than $X$ or $X'$, it must also have no codimension one strata.

Notice that $X_n\cap X'_n$ is dense in both $X_n$ and $X'_n$ since $X_n$ and $X_n'$ are each dense in $X$. 
Now, by definition, an orientation on $X$ is an orientation on $X_n$, i.e. an isomorphism on $X_{n}$ from the orientation $R$-bundle to the constant $R$-bundle. Just as in the proof of  Borel \cite[Lemma 4.11.a]{Bo}, the restriction of this isomorphism to the dense subset  $X_n\cap X'_n$ extends uniquely to an isomorphism from  the orientation $R$-bundle to the constant $R$-bundle on $X'_n$, using the equivalence of local systems with $\pi_1$-modules on connected manifolds (in this case, components of $X_n'$), and that the fundamental group of a dense open set of a connected manifold surjects onto the fundamental group of the connected manifold. 
\end{proof}

\begin{proposition}
If $X$ is a compact $R$-oriented $n$-dimensional stratified pseudomanifold 
and $X'$ is a restratification of $X$ with the induced $R$-orientation then the 
canonical isomorphism $I^{\bar 0}H_n(X;R)\cong I^{\bar 0}H_n(X';R)$ takes 
$\Gamma_X$ to $\Gamma_{X'}$. 
 \end{proposition}
\begin{proof}
This follows easily from Corollary \ref{T: fund id2}, choosing points in $X_n\cap X_n'$. 
\end{proof}

Finally, we observe that the fundamental class is an oriented topological 
invariant. To see what this means, suppose $X, Y$ are compact $R$-oriented 
$n$-dimensional stratified pseudomanifolds and that $f:X\to Y$ is a 
topological homeomorphism (not necessarily stratified). 
The stratification of $X$ induces a restratification 
$Y'$ of $Y$ with $(Y')^i=f(X^i)$ and
the $R$-orientation of $X$ induces an $R$-orientation of $Y'$.

\begin{definition} Let  $X, Y$ be compact $R$-oriented $n$-dimensional stratified pseudomanifolds without codimension one strata. We will say that the topological homeomorphism $f:X\to Y$
is an \emph{oriented homeomorphism} if the induced orientation on $Y'$ is consistent with the given orientation of $Y$ in the sense of Lemma \ref{L: orient}.
\end{definition}

The following corollary is now evident from the preceding results of this subsection:

 \begin{corollary}
\label{x4}
 If $f:X\to Y$ is an oriented homeomorphism of compact $R$-oriented
$n$-dimensional stratified 
pseudomanifolds without codimension one strata, then $f$ takes $\Gamma_X\in 
I^{\bar 0}H_n(X;R)$ to $\Gamma_Y\in I^{\bar 0}H_n(Y;R)$.
 \end{corollary}

\section{Poincar\'e duality}
\label{j4}

Let $F$ be a field.  In this section all intersection homology and cohomology
will have $F$ coefficients.

We will show that cap product with the 
fundamental class induces a Poincar\'e duality isomorphism from compactly
supported intersection cohomology to 
intersection homology.

Let  $X$ be an $F$-oriented stratified pseudomanifold of dimension $n$, 
possibly noncompact and possibly with codimension one strata.
Let $\bar p$ be a perversity. 

\begin{definition}
The compactly supported intersection cohomology of $X$ with perversity 
$\bar p$, denoted $I_{\bar p}H^*_c(X;F)$, is defined to be 
\[
\dlim I_{\bar p}H^*(X,X-K;F),
\]
where $K$ ranges over all compact subsets of $X$.
\end{definition}

Let $\bar q=\bar t-\bar p$.

For each compact $K\subset X$
we define 
\[
{\ms D}_K:I_{\bar p}H^*(X, X-K;F)\to I^{\bar q}H_{n-*}(X;F)
\]
by 
\[
{\ms D}_K(\alpha)=(-1)^{|\alpha|n}(\alpha\smallfrown\Gamma_K).
\]

\begin{remark}
For the sign $(-1)^{|\alpha|n}$, which does not appear in the literature, see
\cite[Section 4.1]{GBF18}, where this sign is introduced to make the duality map a chain map of appropriate degree. 
\end{remark}

Next we observe that the ${\ms D}_K$ are consistent as $K$ varies.  Let
$K\subset K'$ and let $j:X-K'\hookrightarrow X-K$ be the inclusion. 
Then for $\alpha\in I_{\bar p}H^*(X,X-K)$ we have
\begin{align*}
{\ms D}_{K'}(j^*\alpha)
&=
(-1)^{|\alpha|n}((j^*\alpha)\smallfrown\Gamma_K')\\
&=
(-1)^{|\alpha|n}(\alpha\smallfrown(j_*\Gamma_K'))
\quad\text{by Proposition \ref{P: cap inclusion}}\\
&=
(-1)^{|\alpha|n}(\alpha\smallfrown \Gamma_K)
\quad\text{by Remark \ref{j7}}\\
&=
{\ms D}_K(\alpha).
\end{align*}

Now we define
$$\ms D:  I_{\bar p}H_c^*(X;F)\to I^{\bar q}H_{n-*}(X;F)$$ 
to be 
\[
\dlim {\ms D}_K.
\]

\begin{theorem}[Poincar\'e duality]\label{T: duality}
Let $F$ be a field. Let $X$ be an $n$-dimensional $F$-oriented stratified 
pseudomanifold, possibly noncompact and possibly with codimension one strata, 
and let $\bar p+\bar q=\bar t$. Then $\ms D:  I_{\bar p}H_c^i(X;F)\to I^{\bar 
q}H_{n-i}(X;F)$ is an isomorphism\footnote{As this is the central theorem of 
the paper, we remind the reader here that the intersection homology groups of 
this theorem are those defined in Section \ref{S: background}.}.
\end{theorem}

\begin{proof}
We argue by induction on the depth of $X$. When $X$ has depth $0$, $X$ is a
manifold and the result is classical. So we assume now that $X$ has positive
depth and that the theorem has been proven on stratified pseudomanifolds of depth less than that of $X$.

\begin{lemma}\label{L: cone duality}
If the conclusion of Theorem \ref{T: duality} holds for the compact
$F$-oriented stratified $k-1$ pseudomanifold $L$, then it holds for $cL$. 
\end{lemma}

For the proof we need some notation, which will also be used later.

\begin{notation}
For $0<r<1$ let $c_rL$ denote the image of $[0,r]\times L $ in $cL=
([0,1)\times L)/\sim$.
\end{notation}

\begin{proof}
We orient $cL$ consistently with the product $(0,1)\times(L-L^{k-2})\cong 
cL-(cL)^{k-1}$.  Let $v$ denote the vertex of $cL$.

We are free to choose any cofinal collection of compact sets, so we choose 
$K$ to have the form $c_rL$, $0< r< 1$.

Fix such a $K$.  We claim that the map 
\[
{\ms D}_K:
I_{\bar p}H^i(cL, cL-K;F)
\to
I^{\bar q}H_{k-i}(cL;F)
\]
is already an isomorphism (before passage to the direct limit).

Let $b\in (r,1)$ and let $j:L\to cL$ take $x$ to 
$(b,x)$.  Then $j$ is a stratified homotopy equivalence $L\to 
cL-K$, so Appendix A implies that, for every $\bar r$, $j_*$ is an 
isomorphism $I^{\bar r}H_*(L;F)\to I^{\bar r}H_*(cL-K;F)$ and $j$ induces an 
isomorphism $I^{\bar r}H_*(cL,L;F)\to I^{\bar r}H_*(cL,cL-K;F)$.

Now if $i< k-\bar p(\{v\})$ then Proposition \ref{P: cone} and Remark \ref{l1}
imply that the 
the domain and range of ${\ms D}_K$
are both 0 and ${\ms D}_K$ is vacuously an isomorphism.

So let $i\geq k-\bar p(\{v\})$ and
consider the following diagram
\begin{diagram}
I_{\bar p}H^i(cL, cL-K;F)&\lTo^{\delta}& I_{\bar p}H^{i-1}(cL-K;F)&\rTo_\cong^{j^*} & 
I_{\bar p}H^{i-1}(L;F)\\
\dTo^{{\ms D}_K}&&\dTo^{\cdot\capp j_*\Gamma_L}&&\dTo_{\cdot\smallfrown\Gamma_{L}}\\
I^{\bar q}H_{k-i}(cL;F)&\lTo^{\text{inc}}& I^{\bar q}H_{k-i}(cL-K;F)&\lTo_{\cong}^{j_*}&
I_{\bar q}H_{k-i}(L;F).
\end{diagram}

The right vertical arrow is an isomorphism by hypothesis (since $L$ is 
compact), and the right square commutes up to sign by Proposition \ref{P: cap 
inclusion}, so the middle vertical arrow is an isomorphism.
Proposition \ref{P: cone} and Remark \ref{l1} imply that the horizontal arrows
in the left square are isomorphisms so it suffices to show that the left square
commutes up to sign.  For this it suffices, by Proposition \ref{L:
boundary}(2), to show that $j_*\Gamma_L=\bd \Gamma_K$.  This in turn follows
from the fact, shown in the proof of Proposition \ref{P:  
local o}, that $\bar c\Gamma_L= \Gamma_K$
(in that proof it was assumed that $X$ is normal, 
but
the relevant part of the argument holds more generally).
\end{proof}

\begin{lemma}\label{L: product duality}
If the conclusion of Theorem \ref{T: duality} holds for the compact $F$-oriented stratified $k-1$  pseudomanifold $L$, then it holds for $M\times cL$, where $M$ is an $F$-oriented unstratified $n-k$ manifold and we use the product stratification and the product orientation. 
\end{lemma}

\begin{proof}
For convenience, let $M\times cL=Y$.

Any compact set $K\subset Y$ is contained in the compact set $p_1(K)\times
p_2(K)$, where $p_1,p_2$ are the respective projections to $M$ and to
$cL$. So the compact sets of the form $K_1\times K_2\subset Y$ are cofinal
among all compact sets. Furthermore, since compact sets of the form $c_rL$ (in
the notation of the proof of Lemma \ref{L: cone duality}) are cofinal among
compact sets in $cL$, compact sets of the form $K_1\times c_rL$ are cofinal
among the compact sets of $Y$. 
Therefore, to prove the lemma, it suffices to
show that the direct limit of the maps
\[
\cdot\smallfrown\Gamma_{K_1\times c_rL}:I_{\bar p}H^i(Y, Y-(K_1\times 
c_rL);F)\to I^{\bar q}H_{n-i}(Y;F)
\]
is an isomorphism.

Now consider the following diagram. 

{\footnotesize
\begin{diagram}
I_{\bar p}H^*(Y,Y-K_1\times c_rL;F) &\lTo^c&H^*(M,M-K_1;F)\otimes 
I_{\bar p}H^*(cL,cL-c_rL;F)\\
\dTo^{\cdot\smallfrown \Gamma_{K_1\times c_rL}} 
&&\dTo_{(\cdot\smallfrown\Gamma_{K_1})\otimes (\cdot\smallfrown \Gamma_{c_rL})}\\ 
I^{\bar q}H_*(Y;F)&\lTo^\times & H_*(M)\otimes I^{\bar q}H_*(cL).
\end{diagram}
}

Here the map $c$ is defined by $c(\alpha\otimes
\beta)=(-1)^{|\beta|(n-k)}(\alpha\times\beta)$ (recall that the cohomology 
cross product was defined just before Proposition \ref{j6}).  The diagram
commutes by Proposition \ref{j8} and the relative version of Proposition
\ref{j6}.  The lower horizontal arrow is an isomorphism by Theorem \ref{j1} 
(using perversity $\bar 0$ for the $M$ factor) and the upper horizontal arrow 
is an isomorphism by the relative version of Remark \ref{l2}; note that 
$I_{\bar p}H^*(cL,cL-c_rL;F)$ is finitely generated because $L$ is compact\footnote{By \cite{GM2} for Goresky-MacPherson perversities and by \cite[Proposition 4.1]{GBF23} in general, the intersection homology groups are isomorphic to the hypercohomology groups with compact support of cohomologically constructible complexes of sheaves. These groups are thus finitely generated on compact spaces by Wilder's ``Property $(P,Q)$''; see \cite[Section V.3]{Bo} for details.}.  
The right hand vertical arrow induces an isomorphism after passage to the 
direct limit by \cite[Theorem 3.35]{Ha} and Lemma \ref{L: cone duality}.  It 
follows that the left hand vertical arrow induces an isomorphism after 
passage to the direct limit as required.
\end{proof}

We can now complete the proof of Poincar\'e duality on $X$ with a Zorn's Lemma argument, as in the proof of manifold duality in Hatcher \cite[Proof of Theorem 3.35]{Ha}. By the induction assumption, any space of depth less than that of $X$ satisfies the conclusion of the theorem. In particular, it is true on  $X-X^m$ where $X^m$ is the smallest non-empty skeleton of $X$. Let $\mc U$ denote the set of open sets of $X$ containing $X-X^{m}$ and on which $\ms D$ is an isomorphism; $\mc U$ is partially ordered by inclusion. Suppose $S$ is a totally ordered subset of $\mc U$, and let $W=\cup_{U\in S} U$. For $U_a\subset U_b$ elements of $S$, there is a natural map $I_{\bar p}H^i_c(U_a; F)\to I_{\bar p}H^i_c(U_b; F)$ since an element of $I_{\bar p}H^i_c(U_a; F)$ is represented by an element of $I_{\bar p}H^i(U_a, U_a-K; F)$ for some compact $K$ and then $I_{\bar p}H^i(U_a, U_a-K; F)\cong I_{\bar p}H^i(U_b, U_b-K; F)$ by excision. Furthermore, we then see that $\dlim_{U\in S}I_{\bar p}H^i_c(U; F)\cong I_{\bar p}H^i_c(W; F)$. Of course also $I^{\bar q}H_{n-i}(W;F)\cong \dlim_{U\in S}I^{\bar q}H_{n-i}(U;F)$, and it follows that  $\ms D: I_{\bar p}H^i_c(W; F)\to I^{\bar q}H_{n-i}(W; F)$ is the direct limit of duality isomorphisms and hence an isomorphism. 

Therefore, each totally ordered set in $\mc U$ has a maximal element, and by
Zorn's lemma, there is a largest open $U\subset X$ such that $U$ contains
$X-X^{m}$ and duality holds on $U$. If $U=X$ we are done. Suppose $U\neq X$,
and let $x\in X-U$. Then 
$x\in  X-X^{n-k}$
for some $k\geq 1$, and $x$ is contained in a distinguished neighborhood $N$ 
homeomorphic to $\R^{n-k}\times cL^{k-1}$.  From now on we write $X_{n-k}$ for
$X^{n-k}-X^{n-k-1}$.
Proceeding as in the proof of 
\cite[Proposition 8]{Ki}, let $V=U\cap N$. Since this set is open (and so is 
$U\cap N\cap X_{n-k}$ in $X_{n-k}$), we can shrink the $cL$ factors in $N$ to 
obtain an open neighborhood $W$ of $U\cap N\cap X_{n-k}$ in $U\cap N=V$ such 
that $W$ is homeomorphic to $(U\cap N\cap X_{n-k})\times cL$. 

Now we have the following diagram, in which the rows are Mayer-Vietoris 
sequences:
{\footnotesize
\begin{diagram}[LaTeXeqno]\label{dHatcher}
&\rTo& I_{\bar p}H^i_c(W-W\cap X_{n-k};F)&\rTo &I_{\bar p}H^i_c(W;F)\oplus I_{\bar p}H^i_c(V-V\cap X_{n-k};F) &\rTo& I_{\bar p}H^i_c(V;F)&\rTo&\\
&&\dTo^{\ms D}&&\dTo_{\ms D\oplus (-\ms D)}&&\dTo_{\ms D}\\
&\rTo& I^{\bar q}H_{n-i}(W-W\cap X_{n-k};F)&\rTo & I^{\bar q}H_{n-i}(W;F)\oplus I^{\bar q}H_{n-i}(V-V\cap X_{n-k};F) &\rTo& I^{\bar q}H_{n-i}(V;F)&\rTo&.
\end{diagram}}
The diagram commutes up to sign by Proposition \ref{Hatcher} in subsection
\ref{sHatcher}.

The left hand vertical map and the second summand of the middle map are isomorphisms by the induction hypothesis on depth. The first summand of the middle map is an isomorphism by Lemma \ref{L: product duality}.  Hence the right hand map is an isomorphism by the five lemma.

Now we can plug this into the Mayer-Vietoris diagram

\begin{diagram}
&\rTo& I_{\bar p}H^i_c(V;F)&\rTo &I_{\bar p}H^i_c(U;F)\oplus I_{\bar p}H^i_c(N;F) &\rTo
& I_{\bar p}H^i_c(U\cup N;F)&\rTo&\\
&&\dTo&&\dTo&&\dTo\\
&\rTo& I^{\bar q}H_{n-i}(V;F)&\rTo& I^{\bar q}H_{n-i}(U;F)\oplus I^{\bar q}H_{n-i}(N;F) &\rTo& 
I^{\bar q}H_{n-i}(U\cup N;F)&\rTo&,
\end{diagram}
and we conclude similarly that duality holds on $U\cup N$, contradicting the maximality of $U$. Hence we must have $U=X$ and duality holds on $X$.

Note: if we assume that $X^{n-1}$ is second countable, then rather than resort to Zorn's lemma, we could instead use the same diagrams to perform an induction, starting with $X-X^{n-1}$ and then taking unions one at a time with members of a countable covering of $X^{n-1}$ by distinguished neighborhoods. 

\end{proof}

\subsection{Commutativity of diagram \eqref{dHatcher}}
\label{sHatcher}

In this subsection all intersection chain groups and intersection homology 
groups have $F$-coefficients, which will not be included in the notation.  
Our goal is to prove the following analogue of Lemma 3.36 of \cite{Ha}.

\begin{proposition}\label{Hatcher}
Let $X$ be an $F$-oriented stratified pseudomanifold.  Let $U$ and $V$ be open
subsets of $X$ with $X=U\cup V$.  Let $\bar p+\bar q=\bar t$. Then the 
following diagram, in which the rows are Mayer-Vietoris sequences, commutes 
up to sign.
{\footnotesize
\begin{diagram}[LaTeXeqno]\label{h1}
&\rTo& I_{\bar p}H^i_c(U\cap V)&\rTo &I_{\bar p}H^i_c(U)\oplus
I_{\bar p}H^i_c(V) &\rTo& I_{\bar p}H^i_c(X)&\rTo&I_{\bar p}H^{i+1}_c(U\cap 
V)&\rTo&\\
&&\dTo^{\ms D}&&\dTo_{\ms D\oplus (-\ms D)}&&\dTo_{\ms D}&&\dTo^{\ms D}\\
&\rTo& I^{\bar q}H_{n-i}(U\cap V)&\rTo & I^{\bar q}H_{n-i}(U)\oplus 
I^{\bar q}H_{n-i}(V) &\rTo& I^{\bar q}H_{n-i}(X)&\rTo&I^{\bar q}H_{n-i-1}(U\cap V)&\rTo&.
\end{diagram}}
\end{proposition}

Our proof will follow the general strategy of \cite{Ha} (but with our sign
conventions).  As in \cite{Ha}, the
commutativity up to sign of the three squares shown in diagram \eqref{h1} is 
an easy consequence of the three parts of the following lemma.  

\begin{lemma}
\label{h2}
Let $K$ and $L$ be compact subsets of $U$ and $V$. The following diagrams
commute.
\begin{enumerate}
\item
{\footnotesize
\begin{diagram}
& I_{\bar p}H^k(X,X-K\cap L )&\rTo &I_{\bar p}H^k(X,X-K )\oplus I_{\bar p}H^k(X,X-L )&\\
&\dTo&&\dTo&&\\
& I_{\bar p}H^k(U\cap V, U\cap V-K\cap L )& &I_{\bar p}H^k(U,U-K )\oplus
I_{\bar p}H^k(V,V-K )&\\
&\dTo_{\smallfrown \Gamma_{K\cap L}}&&
\dTo_{\smallfrown\Gamma_K\oplus-(\smallfrown\Gamma_L)}&\\
& I^{\bar q}H_{n-k}(U\cap V )&\rTo &I^{\bar q}H_{n-k}(U )\oplus I^{\bar q}H_{n-k}(V )&
\end{diagram}}
\item
{\footnotesize
\begin{diagram}
&I_{\bar p}H^k(X,X-K )\oplus I_{\bar p}H^k(X,X-L )&\rTo&I_{\bar p}H^k(X,X-K\cup L )&\\
&\dTo&&\\
&I_{\bar p}H^k(U,U-K )\oplus
I_{\bar p}H^k(V,V-K )&&\dTo_{\smallfrown \Gamma_{K\cup L}}&\\
&\dTo_{\smallfrown\Gamma_K\oplus -(\smallfrown\Gamma_L)}&\\
&I^{\bar q}H_{n-k}(U )\oplus I^{\bar q}H_{n-k}(V )&\rTo&I^{\bar q}H_{n-k}(X )&
\end{diagram}}
\item
\begin{diagram}[LaTeXeqno]\label{h4}
I_{\bar p}H^k(X,X-K\cup L)&\rTo^\delta&I_{\bar p}H^{k+1}(X,X-K\cap
L)&\rTo&I_{\bar p}H^{k+1}(U\cap V,U\cap V-K\cap L)\\
\dTo_{\smallfrown\Gamma_{K\cup L}}&&&&\dTo_{\smallfrown\Gamma_{K\cap L}}\\
I^{\bar q}H_{n-k}(X)&&\rTo^\bd&&I^{\bar q}H_{n-k-1}(U\cap V).
\end{diagram}

\end{enumerate}
\end{lemma}

In the remainder of this section we prove Lemma \ref{h2}.

For part 1, it suffices to consider the two summands on the right hand side
separately.  We will verify commutativity for the first summand; the second is
similar.  Consider the following diagram, where all unmarked arrows are 
induced by inclusions.
\[
\xymatrix{
I_{\bar p}H^k(X,X-K\cap L )
\ar[d]
\ar[rr]
\ar[rd]
&&
I_{\bar p}H^k(X,X-K )
\ar[d]
\\
I_{\bar p}H^k(U\cap V, U\cap V-K\cap L )
\ar[d]^{\smallfrown \Gamma_{K\cap L}}
&
I_{\bar p}H^k(U,U-K\cap L )
\ar[l]
\ar[r]
\ar[rd]_{\smallfrown \Gamma_{K\cap L}}
&
I_{\bar p}H^k(U,U-K )
\ar[d]^{\smallfrown\Gamma_K}
\\
I^{\bar q}H_{n-k}(U\cap V )
\ar[rr]
&&I^{\bar q}H_{n-k}(U )
}
\]
Here the upper half obviously commutes, and the lower half commutes by
Proposition \ref{P: cap inclusion} (using the fact that the inclusion
$(U,U-K)\to (U,U-K\cap L)$ takes $\Gamma_K$ to $\Gamma_{K\cap L}$).

For part 2, it again suffices to work one summand at a time.  For the first
summand, consider the following diagram.
\[
\xymatrix{
I_{\bar p}H^k(X,X-K )
\ar[d]
\ar[r]
\ar[rdd]^{\smallfrown \Gamma_K}
&
I_{\bar p}H^k(X,X-K\cup L )
\ar[dd]^{\smallfrown \Gamma_{K\cup L}}
\\
I_{\bar p}H^k(U,U-K )
\ar[d]^{\smallfrown\Gamma_K}
&
\\
I^{\bar q}H_{n-k}(U )
\ar[r]
&
I^{\bar q}H_{n-k}(X )
}
\]
Both triangles commute by Proposition \ref{P: cap inclusion}, using the fact
that the inclusion
$(X,X-K\cup L)\to (X,X-K)$ takes $\Gamma_{K\cup L}$ to $\Gamma_K$.

Next we prove part 3.  We need a lemma which will be proved at the end of
this subsection.  

\begin{lemma}\label{h3}
There exist chains
\[
\beta_{U-L}\in I^{\bar p}C_*(U-L)\otimes I^{\bar q}C_*(U-L, U-K\cup L),
\]
\[
\beta_{U\cap V}\in I^{\bar p}C_*(U\cap V)\otimes I^{\bar q}C_*(U\cap V, U\cap
V-K\cup L)
\]
and
\[
\beta_{V-K}\in I^{\bar p}C_*(V-K)\otimes I^{\bar q}C_*(V-K,V-K\cup L)
\]
such that $\beta_{U-L}+\beta_{U\cap V}+\beta_{V-K}$ represents
${\bar d}(\Gamma_{K\cup L})\in I^{\bar p}H_*(X)\otimes I^{\bar q}H_*(X,X-K\cup
L)
$.
\end{lemma}

The inclusion $(X,X-K\cup L)\to (X,X-K\cap L)$ takes $\Gamma_{K\cup L}$ to
$\Gamma_{K\cap L}$, so the image of $\beta_{U-L}+\beta_{U\cap V}+\beta_{V-K}$
in $I^{\bar p}H_*(X)\otimes I^{\bar q}H_*(X,X-K\cap L)$ represents 
${\bar d}(\Gamma_{K\cap L})$.  But this image is just $\beta_{U\cap V}$, since the
other two terms map to $0$ in $I^{\bar p}C_*(X)\otimes I^{\bar q}C_*(X,X-K\cap
L)$.  Thus $\beta_{U\cap V}$ represents the class 
${\bar d}(\Gamma_{K\cap L})$ in $I^{\bar p}H_*(U\cap V)\otimes I^{\bar q}H_*(U\cap 
V, U\cap V-K\cup L)$.

Now let $\varphi\in I_{\bar p}C^k(X,X-K\cup L)$ be a cocycle; we want to 
calculate the image of $[\varphi]$ for the two ways of going around the 
diagram \eqref{h4}.  Let $A$ and $B$ denote $X-K$ and $X-L$, so that 
$\varphi\in I_{\bar p}C^k(X,A\cap B)$.  

As in \cite{Ha}, we 
have $\delta[\varphi]=[\delta\varphi_A]$, where 
$\varphi=\varphi_A-\varphi_B$ with $\varphi_A\in I_{\bar p}C^k(X,A)$ and 
$\varphi_B\in I_{\bar p}C^k(X,B)$. Continuing on to $I^{\bar q}H_{n-k-1}(U\cap
V)$ we obtain $[(1\otimes \delta\varphi_A)(\beta_{U\cap V})]$, which is the
same as
\[
(-1)^{k+1}[(1\otimes\varphi_A)(\partial\beta_{U\cap V})]
\]
because $\partial((1\otimes\varphi_A)(\beta_{U\cap V}))=
(1\otimes\delta\varphi_A)(\beta_{U\cap V})
+(-1)^k(1\otimes \varphi_A)(\partial\beta_{U\cap V})$.

Going around the diagram \eqref{h4} the other way, let $\beta$ denote
$\beta_{U-L}+\beta_{U\cap V}+\beta_{V-K}$.
$[\varphi]$ first maps to $[(1\otimes\varphi)(\beta)]$.  To apply the
Mayer-Vietoris boundary $\partial$ to this, we first write
$(1\otimes\varphi)(\beta)$ as a sum of a chain in $U$ and a chain in $V$:
\[
(1\otimes\varphi)(\beta)
=
(1\otimes\varphi)(\beta_{U-L})
+
((1\otimes\varphi)(\beta_{U\cap V})
+
(1\otimes\varphi)(\beta_{V-K})).
\]
Then we take the boundary of the first of these two terms, obtaining the
homology class
$[\partial(1\otimes\varphi)(\beta_{U-L})]$.  To compare this to 
$(-1)^{k+1}[(1\otimes\varphi_A)(\partial\beta_{U\cap V})]$, we have
\begin{alignat*}{2}
\partial(1\otimes\varphi)(\beta_{U-L})
&=(-1)^k (1\otimes \varphi)(\partial\beta_{U-L})
&&\qquad\text{since $\delta\varphi=0$}
\\
&=(-1)^k(1\otimes\varphi_A)(\partial\beta_{U-L})
&&\qquad\text{since $(1\otimes\varphi_B)(\partial\beta_{U-L})=0$, $\varphi_B$
being}
\\
&&&\qquad\text{zero on chains in $B=X-L$}
\\
&=(-1)^{k+1}(1\otimes\varphi_A)(\partial\beta_{U\cap V}),
&&
\end{alignat*}
where the last equality comes from the fact that
$\partial(\beta_{U-L})+\partial(\beta_{U\cap V})=\bd(\beta)-\partial(\beta_{V-K})$
and $\varphi_A$ vanishes on chains in $V-K\subset A$.

This concludes the proof of Lemma \ref{h2}.  \hfill\qedsymbol

It remains to prove Lemma \ref{h3}.

Let $\cC$ be the category with objects $U-L$, $U\cap V$, $V-K$ and their
intersections and with morphisms the inclusion maps. 
It suffices to show that $\bar{d}(\Gamma_{K\cup L})$ is in the image of the map
\[
\kappa:\dlim_{W\in \cC}\, I^{\bar p}H_*(W)\otimes I^{\bar q}H_*(W, W-K\cup L)
\to
I^{\bar p}H_*(X)\otimes I^{\bar q}H_*(X,X-K\cup L).
\]
Let $Y$ denote the subspace
\[
((U-L)\times (U-L))\cup
((U\cap V)\times (U\cap V))\cup
((V-K)\times (V-K))
\]
of $X\times X$ and
consider the commutative diagram
\[
\scriptsize
\xymatrix{
I^{\bar 0}H_*(X,X-K\cup L)
\ar[r]^-d
\ar[rd]^d
&
I^{Q_{\bar p,\bar q}}H_*(X\times X, X\times (X-K\cup L))
&
I^{\bar p}H_*(X)\otimes I^{\bar q}H_*(X,X-K\cup L)
\ar[l]_-\cong
\\
&
I^{Q_{\bar p,\bar q}}H_*(Y,Y-(X\times (K\cup L)))
\ar[u]
&
\\
&
H_*(\dlim_{W\in \cC}\, I^{Q_{\bar p,\bar q}}C_*(W\times W,W\times (W-K\cup L)))
\ar[u]^\lambda
&
H_*(\dlim_{W\in \cC}\, I^{\bar p}C_*(W)\otimes I^{\bar q}C_*(W, W-K\cup L)).
\ar[uu]^\kappa
\ar[l]_-\mu
}
\]
$\bar{d}(\Gamma_{K\cup L})$ is the image of $\Gamma_{K\cup L}$ along the top
row.  The map $\lambda$ is an isomorphism by 
\cite[Proposition 6.1.1]{GBF31}, so to show that $\bar{d}(\Gamma_{K\cup L})$ is
in the image of $\kappa$ it suffices to show that the map $\mu$ is an
isomorphism.

Let us write $W_1$, $W_2$ and $W_3$ for $U-L$, $U\cap V$ and $V-K$
respectively.
Let $\cC'$ be the subcategory of $\cC$ with objects $W_1$, $W_2$ and $W_1\cap
W_2$, and let $\cC''$ be the subcategory of $\cC$ with objects $W_1\cap W_3$,
$W_2\cap W_3$ and $W_1\cap W_2\cap W_3$.  For any functor $F$ from $\cC$ to
chain complexes,
$\dlim_{W\in \cC} F(W)$ can be written as an iterated pushout: it is the pushout of the 
diagram
\[
\xymatrix{
\dlim_{W\in \cC''} F(W)
\ar[d]
\ar[r]
&
F(W_3)
\\
\dlim_{W\in\cC'} F(W)
}
\]
and $\dlim_{W\in\cC''} F(W)$ and $\dlim_{W\in\cC'} F(W)$ are also pushouts.

Next recall that, if
\[
\xymatrix{
A
\ar[r]
\ar[d]
&
C
\ar[d]
\\
B
\ar[r]
&
D
}
\]
is a pushout diagram of chain complexes for which $A\to B\oplus C$ is a
monomorphism, there is a Mayer-Vietoris sequence
\[
\cdots\to H_iA\to H_iB\oplus H_iC \to H_iD \to H_{i-1}A\to\cdots
\]
Combining this with Theorem \ref{P: relative kunneth} and the five lemma, we
see that the map
\[
\dlim_{W\in \cC'}\, I^{\bar p}C_*(W)\otimes I^{\bar q}C_*(W, W-K\cup L)
\to
\dlim_{W\in \cC'}\, I^{Q_{\bar p,\bar q}}C_*(W\times W, W\times (W-K\cup L)),
\]
and the analogous map with $\cC'$ replaced by $\cC''$, are quasi-isomorphisms.
Now one further application of the Mayer-Vietoris sequence, Theorem \ref{P:
relative kunneth}, and the five lemma shows that $\mu$ is an
isomorphism as required.
\hfill\qedsymbol

\section{Stratified pseudomanifolds-with-boundary and Lefschetz duality}\label{S: boundary}

In subsection \ref{j9}, we give the definition of stratified
pseudomanifold-with-boundary; we call these $\partial$-stratified
pseudomanifolds, following Dold's use of $\partial$-manifold to mean manifold
with boundary \cite[Definition VIII.1.9]{Dold}.  
In subsection \ref{fund boundary}, we show that a compact $\partial$-stratified
pseudomanifold has a fundamental class, and in subsection \ref{Lef} we show
that cap product with the fundamental class induces a Lefschetz duality
isomorphism.

\subsection{$\bd$-stratified pseudomanifolds}
\label{j9}

\begin{definition}\label{def boundary}
An $n$-dimensional
\emph{$\bd$-stratified pseudomanifold} is a pair $(X,B)$ together with a
filtration on $X$ such that
\begin{enumerate}
\item $X-B$, with the induced filtration, is an $n$-dimensional stratified
pseudomanifold (in the sense of Section \ref{S: pm}),
\item $B$, with the induced filtration, is an $n-1$ dimensional stratified
pseudomanifold (in the sense of Section \ref{S: pm}),
\item\label{I: collar} $B$ has an {\it open collar neighborhood} in $X$, that is, a
neighborhood $N$ with a homeomorphism of filtered spaces $N\to
B\times [0,1)$ (where $[0,1)$ is given the trivial filtration) that takes
$B$ to $B\times \{0\}$.
\end{enumerate}
$B$ is called the 
\emph{boundary} of $X$ and will be denoted by $\bd X$.  
\end{definition}

We will often abuse notation by referring to the ``$\bd$-stratified 
pseudomanifold $X$,'' leaving $B$ tacit. 

Note that a stratified pseudomanifold $X$ (as defined in Section \ref{S: pm}) is a $\bd$-stratified pseudomanifold with $\bd X=\emptyset$.  

\begin{definition}
The {\it strata} of a $\partial$-stratified pseudomanifold $X$ are the 
components of the spaces $X^i-X^{i-1}$.
\end{definition}

There is some risk of confusion between the situation in which a stratified pseudomanifold has a codimension one stratum and the situation in which a $\bd$-stratified pseudomanifold has a non-empty boundary. To understand the difference between these situations, it is critical to note that the choice of stratification of the underlying space plays a role. In particular it is possible for the same topological space to possess different stratifications, some with respect to which it has a non-empty boundary and some with respect to which it does not.  We illustrate this situation with the following elementary example:

\begin{example}\label{E: example}
Let $M$ be a compact  $n$-manifold with boundary (in the classical sense), and let $P$ be its boundary.
\begin{enumerate}
\item
Suppose we filter $M$ trivially so that $M$ itself is the only non-empty stratum, i.e. the filtration is $\emptyset \subset M$. Then 
$(M,P)$ is a $\bd$-stratified pseudomanifold. Note that all the conditions of Definition \ref{def boundary} are fulfilled: $M-P$ is an $n$-manifold, $P$ is an $n-1$ manifold, and $P$ is collared in $M$ by classical manifold theory (see \cite[Proposition 3.42]{Ha}).

\item
On the other hand, suppose we take the same manifold $M$ and now give it the 
filtration $P\subset M$. Let us denote the space together with its filtration 
information by $X$. Then it is easy to check that $X$ is a stratified
pseudomanifold in the sense of Section \ref{S: pm}; equivalently, $X$ is a
$\partial$-stratified pseudomanifold with $\partial X=\emptyset$.
\end{enumerate}
\end{example}

In this example one can restratify a $\bd$-stratified pseudomanifold to
get a stratified pseudomanifold.  We do not know of an example where this
is not possible.
On the other hand, if codimension one strata are not allowed then
the following result shows that the boundary of a $\partial$-stratified 
pseudomanifold is an invariant of the underlying topological space.

\begin{proposition}\label{P: bd inv}
Let $(X,B)$ and $(X',B')$ be $\partial$-stratified pseudomanifolds of dimension
$n$ with no codimension one strata, and let $h:X\to X'$ be a homeomorphism 
(which is not required to be filtration preserving).  Then $h$ takes $B$ onto 
$B'$.
\end{proposition} 

\begin{proof}
It suffices to show that $h$ takes $B$ to 
$B'$, as the equivalent result for $h^{-1}$ shows then that $h$ takes $B$ \emph{onto} $B'$. 

It thus suffices to show that $h$ takes the union of the regular strata of 
$B$ to
$B'$, since the regular strata are dense in $B$ and $B'$ is closed. So let
$x$ be in a regular stratum of $B$ and suppose that $h(x)$ is not in $B'$.
Then there is a Euclidean neighborhood $E$ of $x$ in $B$ such that
$h(E)\subset X'-B'$.  The existence of an open collar neighborhood of $B$
shows that the local homology group $H_n(X,X-\{y\})$ is 0 for each $y\in E$,
so by topological invariance of homology $h(E)$ must be contained in the
singular set $S$ of $X'-B'$.  

Next we use the dimension theory of \cite[Section II.16]{BR}.  We will use the
fact that each skeleton of a pseudomanifold (and in particular the singular 
set) is locally compact.

$\dim_\Z E$ (as defined in \cite[Definition II.16.6]{BR}) is $n-1$ by
\cite[Corollary II.16.28]{BR}, so $\dim_\Z h(E)$ is also $n-1$, and 
by \cite[Theorem II.16.8]{BR}
(using the fact that $S$
is locally compact) this implies that
$\dim_\Z S$ is $\geq n-1$.  To obtain a contradiction it suffices to show
that $\dim_\Z$ of the $i$-skeleton of a pseudomanifold is $\leq i$ (a
fact that doesn't seem to be written down explicitly in the literature).

So let $Y$ be a pseudomanifold and 
assume by induction that $\dim_\Z Y^i\leq 
i$ for some $i$.
Let $c$ denote the family of compact supports
and let $\dim_{c,\Z}$ be as in \cite[Definition 16.3]{BR}. Then $\dim_\Z$ 
is equal to $\dim_{c,\Z}$ for any locally compact space 
by \cite[Definition II.16.6]{BR}. 
Since $Y^i$ is a closed subset of $Y^{i+1}$ and
$Y^{i+1}-Y^i$ is a (possible empty) $(i+1)$-manifold, 
\cite[Exercise II.11 and Corollary II.16.28]{BR} imply that $\dim_{c,\Z}
Y^{i+1}$ is $\leq i+1$ as required.
\end{proof}

\begin{remark}
All of the intersection homology machinery developed in Sections \ref{S: 
background}--\ref{j3} of this paper applies immediately to $\bd$-stratified 
pseudomanifolds. 
\end{remark}

\subsection{Fundamental classes of $\bd$-stratified pseudomanifolds}
\label{fund boundary}

\begin{definition}
An {\em $R$-orientation} of a $\bd$-stratified pseudomanifold $X$ is an
$R$-orientation of $X-\bd X$.
\end{definition}

Given an $R$-orientation of $X$ and a point $x\in X-\bd X$, Definition 
\ref{D: non-normal local class} gives a local orientation class $o_x\in 
I^{\bar 0}H_n(X-\bd X,X-\{x\}-\bd X;R)$.  We will denote the image of this
class under the inclusion map $I^{\bar 0}H_n(X-\bd X,X-\{x\}\cup\bd X;R)\to
I^{\bar 0}H_n(X,X-\{x\};R)$ by $o'_x$.

\begin{proposition}\label{x5}
Let $X$ be a compact $R$-oriented $\bd$-stratified pseudomanifold of 
dimension $n$.  There
is a unique class $\Gamma_X\in I^{\bar 0}H_n(X,\bd X;R)$ that restricts to
$o'_x$ for every $x\in X-\bd X$.
\end{proposition}

\begin{proof}
Let $N$ be an open collar neighborhood of $\bd X$.  Theorem \ref{T: global o}
gives a fundamental class $\Gamma_{X-N}$ in $I^{\bar 0}H_n(X-\bd X,N-\bd X;R)$.
Let $\Gamma_X$ be the image of $\Gamma_{X-N}$ under the composite
\[
I^{\bar 0}H_n(X-\bd X,N-\bd X;R)\to
I^{\bar 0}H_n(X,N;R)
\xleftarrow{\cong} 
I^{\bar 0}H_n(X,\bd X;R),
\]
where the second map (which is induced by inclusion) is an isomorphism by a stratified
 homotopy equivalence (see Appendix \ref{A: sphe}).  It is easy to check that $\Gamma_X$ 
is independent of $N$, using the fact that the intersection of two open collar
neighborhoods contains another.  If $x\in X-\bd X$, the fact that $\Gamma_X$
restricts to $o'_x$ follows from the fact that there is an $N$ not containing
$x$.  Uniqueness follows from the uniqueness property in Theorem \ref{T: 
global o} and the fact that the
maps 
$I^{\bar 0}H_n(X-\bd X,X-\{x\}\cup\bd X;R)\to I^{\bar 0}H_n(X,X-\{x\};R)$
and
$I^{\bar 0}H_n(X-\bd X,N-\bd X;R)\to
I^{\bar 0}H_n(X,N;R)$
are isomorphisms by excision.
\end{proof}

$\Gamma_X$ will be called the {\it fundamental class} of $X$.

\begin{remark}
Corollary \ref{x4} has an analogue for $\bd$-stratified pseudomanifolds.  We 
will not give details here because this fact is not needed for our work.
\end{remark}

We conclude this section with a result that will be needed in \cite{GBF31}.

First we observe that an $R$-orientation of $X$ induces an $R$-orientation of
$\bd X$, because the union of the regular strata of $X$ and the regular strata
of $\partial X$ is a (nonsingular) $R$-oriented $\bd$-manifold.

\begin{proposition}
Let $X$ be a compact $R$-oriented $\bd$-stratified pseudomanifold of dimension
$n$, and give 
$\bd X$ the induced orientation.  Then the map
\[
\bd: I^{\bar 0}H_n(X,\bd X;R)\to I^{\bar 0}H_{n-1}(\bd X;R)
\]
takes $\Gamma_X$ to $\Gamma_{\bd X}$.
\end{proposition}

\begin{proof}
By Corollary \ref{T: fund id2}, it suffices to show that $\bd \Gamma_X$
restricts to the local orientation class in $I^{\bar 0}H_{n-1}(\bd X,\bd
X-\{x\};R)$ for each $x$ that's in a regular stratum of $\bd X$.  So let $x$
be such a point.
Let $E$ be
a closed Euclidean ball around $x$ in $\bd X$, and let $E^\circ$ be the interior of $E$.
Let $N$ be an open collar neighborhood of $\bd X$, and 
let $M$ be the image of $E\times [0,1/2]$ under the homeomorphism $\bd X\times
[0,1)\to N$; then $M$ is a (nonsingular) $\bd$-manifold and the $R$-orientation
of $X$ restricts to an $R$-orientation of $M$.  Let $M^\circ$ denote the
interior of $M$.  Now consider the following commutative diagram (where the 
$R$ coefficients are
tacit).
\[
\scriptsize{
\xymatrix{
I^{\bar 0}H_n(X,\bd X)
\ar[dd]_\bd
\ar[r]
&
I^{\bar 0}H_n(X,X-M^\circ)
\ar[d]_\bd
&
I^{\bar 0}H_n(M,\bd M)
\ar[l]_-\cong
\ar[d]_\bd
&
\\
&
I^{\bar 0}H_{n-1}(X-M^\circ)
\ar[d]
&
I^{\bar 0}H_{n-1}(\bd M)
\ar[d]
\ar[l]
&
\\
I^{\bar 0}H_{n-1}(\bd X)
\ar[r]
\ar[ru]
&
I^{\bar 0}H_{n-1}(X-M^\circ,X-(M^\circ\cup E^\circ))
&
I^{\bar 0}H_{n-1}(\bd M,\bd M-E^\circ)
\ar[l]_-\cong
\ar[r]
&
I^{\bar 0}H_{n-1}(E^\circ,E^\circ-\{x\}).
}
}
\]
Here the second arrows in the first and last rows (which are induced by
inclusion) are isomorphisms by a
combination of excision and stratified homotopy equivalence.
It's straightforward to check that the lower composite is the usual restriction
map $I^{\bar 0}H_{n-1}(\bd X) \to I^{\bar 0}H_{n-1}(E^\circ,E^\circ-\{x\})$, so
it suffices to show that this composite takes $\bd\Gamma_X$ to the local
orientation class at $x$.  But it's straightforward to check that the upper
composite takes $\Gamma_X$ to $\Gamma_M$, and a standard fact in manifold
theory (using the fact that $I^{\bar 0}H_*=H_*$ for spaces with trivial
stratification) says that the rightmost $\bd$ takes $\Gamma_M$ to 
$\Gamma_{\bd M}$.  Since $\Gamma_{\bd M}$ maps to the local orientation class
at $x$ the proof is complete.
\end{proof}

\subsection{Lefschetz duality}
\label{Lef}

\begin{theorem}[Lefschetz Duality]
Let $F$ be a field, and let $X$ be an $n$-dimensional compact $F$-oriented
$\bd$-stratified pseudomanifold. Suppose that
$\bar p +\bar q=\bar t$. 
Then the 
cap product with $\Gamma_X$ gives isomorphisms
\[
I_{\bar p}H^i(X,\bd X;F)\to 
I^{\bar q}H_{n-i}(X;F)
\]
and
\[
I_{\bar p}H^i(X;F)\to 
I^{\bar q}H_{n-i}(X,\bd X;F)
\]
\end{theorem}
\begin{proof}
We follow the strategy in \cite{Ha}.
For the first isomorphism,
let $N$ be an open collar of $\bd X$. Consider the following commutative diagram
\begin{diagram}
I_{\bar p}H^i(X-\bd X,N-\bd X;F)&\lTo^{\cong} &I_{\bar p}H^i(X,N;F)\\
\dTo^{(-1)^{in}\cdot\smallfrown \Gamma_{X-N}}&&\dTo^{(-1)^{in}\cdot\smallfrown \Gamma_X}\\
I^{\bar q}H_{n-i}(X-\bd X;F)&\rTo^{\cong} &I^{\bar q}H_{n-i}(X;F).
\end{diagram}
The top map is an isomorphism by excision and
the bottom by stratified homotopy equivalence. If we take the direct limit of 
the diagram as $N$ shrinks to 
$\bd X$, then $\dlim I_{\bar p}H^i(X,N;F)\cong I_{\bar p}H^i(X,\bd X;F)$ (in 
fact, all maps in the directed system obtained by retracting the collar are 
isomorphisms), while $\dlim I_{\bar p}H^i(X-\bd X,N-\bd X;F)\cong I_{\bar p}H^i_c(X-\bd X;F)$. So by Theorem \ref{T: duality}, the left hand map becomes 
an isomorphism in the limit. It follows therefore that the right hand map also 
becomes an isomorphism in the limit, as required.

For the second isomorphism, we use the diagram (with $F$-coefficients tacit)
\[
\xymatrix{
I_{\bar p}H^{i-1}\partial X
\ar[r]
\ar[d]_{\smallfrown \Gamma_{\partial X}}
&
I_{\bar p}H^i(X,\partial X)
\ar[r]
\ar[d]_{\smallfrown \Gamma_X}
&
I_{\bar p}H^i X
\ar[r]
\ar[d]_{\smallfrown \Gamma_X}
&
I_{\bar p}H^i\partial X
\ar[r]
\ar[d]_{\smallfrown \Gamma_{\partial X}}
&
I_{\bar p}H^{i+1}(X,\partial X)
\ar[d]_{\smallfrown \Gamma_X}
\\
I^{\bar q}H_{n-i}\partial X
\ar[r]
&
I^{\bar q}H_{n-i}X
\ar[r]
&
I^{\bar q}H_{n-i}(X,\partial X)
\ar[r]
&
I^{\bar q}H_{n-i-1}\partial X
\ar[r]
&
I^{\bar q}H_{n-i-1}X
}
\]
The diagram commutes up to sign by Propositions 
\ref{P: cap inclusion} and \ref{L: boundary}.  The result now follows from the
five lemma.
\end{proof}

\appendix

\section{Stratified maps, homotopy, and homotopy equivalence}\label{A: sphe}

The definition of ``stratum preserving homotopy equivalence'' given in
\cite{GBF3, Q1} needs to be modified a little in the context of general
perversities.  In this appendix we give the necessary details.

Let $X$ and $Y$ be $\bd$-stratified pseudomanifolds, and assume that we are
given perversities $\bar p,\bar q$ on $X$ and $Y$ respectively.

\begin{definition}\label{D: maps}
We will say that a map
$f:X\to Y$ 
is \emph{stratified with respect to $\bar p,\bar q$} if 
\begin{enumerate}
\item \label{I: maps} the image of each stratum of $X$ is contained in a single stratum of $Y$ of the same codimension, i.e. if $Z'\subset Y$ is a stratum of codimension $k$, then $f^{-1}(Z')$ is a union of strata of $X$ of codimension $k$,
\item if the stratum $Z\subset X$ maps to the stratum $Z'\subset Y$, then $\bar p(Z)\leq \bar q(Z')$. 
\end{enumerate}
\end{definition}

Note that if $f:X\to Y$ is an inclusion of an open subset, then $f$ is always 
stratified with respect to any perversity $\bar q$ on $Y$ and its induced 
restriction to $X$ (i.e. the perversity on $X$ whose value on $Z$ is  defined 
to be $\bar q(Z')$ if $Z\subset Z'$).  

An easy argument from the definitions shows that if $f:X\to Y$ is stratified and  $\mc G$ is a coefficient system on $Y-Y^{\dim(Y)-1}$, then $f_\#:I^{\bar p}C_*(X;f^*\mc G)\to I^{\bar q}C_*(Y;\mc G)$ is well-defined and induces a map of intersection homology groups $f_*:I^{\bar p}H_*(X;f^*\mc G)\to I^{\bar q}H_*(Y;\mc G)$.

Now stratify 
$X\times I$ by letting the strata 
have the form $Z\times I$, where $Z$ is a stratum of $X$. This 
stratification induces a natural bijection $Z \leftrightarrow Z\times I$ 
between the singular strata of $X$ and those of $X\times I$ and thus a 
natural bijection of perversities such that a perversity of $X$ corresponds 
to a perversity of $X\times I$ if the two take the same value on 
corresponding singular strata. In this case we will abuse notation and use 
the same symbol for both 
perversities.

We call $F: X\times I\to Y$ a \emph{stratified homotopy (with respect to 
$\bar p, \bar q$)}  if $F$ is a stratified map (with respect to $\bar p, \bar q$). In particular, the image under $F$ of each stratum $Z\times I\subset 
X\times I$ is contained in a single stratum of $Y$ (again compare 
\cite{GBF3, Q1}). If $F:X\times I\to Y$ is a stratified homotopy, then $f=F(\cdot,0)$ and 
$g=F(\cdot,1)$ are stratified maps $X\to Y$ and $F$ induces a chain homotopy 
between the induced maps of intersection chains $f_\#$ and $g_\#$. The proof 
of this fact follows by the usual prism construction (see e.g. \cite{Ha}). 
One checks that the necessary chains are all allowable as in the proof of 
Proposition 2.1 of \cite{GBF3}, with some obvious changes necessary to 
account for the general perversities.

We call $\bd$-stratified pseudomanifolds $X,Y$ \emph{stratified homotopy 
equivalent} if there is a homotopy equivalence $f:X\to Y$ with homotopy inverse $g:Y\to X$ such that 
$f$, $g$, and the respective homotopies  from $fg$ to $\text{id}_Y$ and from $gf$ to $\text{id}_X$ all satisfy condition \eqref{I: maps} of Definition \ref{D: maps}. 
 The maps $f$ and $g$ are then deemed \emph{stratified homotopy 
equivalences}. In this case, there must be a bijection between the strata of 
$X$ and the strata of $Y$, and thus a bijection between perversities on $X$ and perversities 
on $Y$. We often abuse notation and use a common symbol for 
the corresponding perversities. With respect to such corresponding 
perversities, $f$ and $g$ will be stratified maps, and the homotopies from $fg$ to $\text{id}_Y$ and from $gf$ to $\text{id}_X$ will be stratified homotopies. 

Thus if $f: X \to Y$ is a stratified homotopy equivalence, it follows that 
$I^{\bar p}C_*(X;f^*\mc G)$ is chain homotopy equivalent to 
$I^{\bar p}C_*(Y;\mc G)$ and thus  $I^{\bar p}H_*(X;f^*\mc G)\cong I^{\bar p}H_*(Y;\mc G)$; see \cite[Lemma 2.4]{GBF10} and \cite[Section 2]{GBF3}. In particular, any inclusion $X\times \{t\}\into X\times I$, 
where $I$ is unfiltered and $X\times I$ is given the product filtration, 
induces $I^{\bar p}H_*\left(X\times \{t\};\mc G|_{X\times \{t\}}\right)\cong 
I^{\bar p}H_*(X\times I;\mc G)$.

\section{Proofs of Theorems \ref{j1} and \ref{P: relative kunneth}}
\label{Appendix A}

In this appendix, we provide some technical proofs 
concerning the intersection homology K\"unneth theorem of \cite{GBF20}. The notation is taken from \cite{GBF20}; we refer the reader there for discussion of the sheaves involved.

We first prove the following proposition, which implies Theorem \ref{j1}.

\begin{proposition}
\label{l9}
Let $F$ be a field. Then the K\"unneth isomorphism of \cite{GBF20} is induced (up to sign) by the chain level cross product.
\end{proposition}
\begin{proof}
Consider the following diagram (all coefficients are in $F$)
\begin{diagram}[LaTeXeqno]\label{D: crossproduct}
H_{n+m-*}(I^{\bar p}C_*(X)\otimes I^{\bar q}C_*(Y))&\rTo^\times & H_{n+m-*}(I^{Q_{\bar p,\bar q}}(X\times Y))\\
\dTo&&\dTo^\cong\\
H^*(\Gamma_c(X\times Y;\mc R^*))  &\rTo& H^*(\Gamma_c(X\times Y; \mc I^{Q_{\bar p,\bar q}} \mc S^*))\\
\dTo&&\dTo^{\cong}\\
H^*(\Gamma_c(X\times Y;\mc I^*))&\rTo^{\cong} &H^*(\Gamma_c(X\times Y;\mc J^*)). 
\end{diagram}
The sheaf $\mc R^*$ is the sheaf defined in \cite{GBF20}; it is really just the sheaf $\pi_X^*(\mc I^{\bar p}\mc S^*_X)\otimes\pi_Y^*(\mc I^{\bar q}\mc S^*_Y)$. The top map is the chain cross product, which is allowable by \cite{GBF20}.
The top vertical maps are induced by sheafification. The bottom vertical maps
are induced by taking $c$-acyclic resolutions; therefore the diagram commutes (up to possible signs arising from the degree shifts in the upper vertical maps).
The maps on the right are isomorphisms by the properties of the sheaf $\mc
I^{Q_{\bar p,\bar q}} \mc S^*$, which is homotopically fine and generated
by a monopresheaf that is conjunctive for coverings\footnotemark. The bottom isomorphism is
the K\"unneth isomorphism of \cite{GBF20}. We want to show that $\times$ is an isomorphism. It suffices to  show that the composition on the left of the diagram is an isomorphism. 

\footnotetext{Note that being a monopresheaf that is conjunctive for coverings is not quite the same thing as being a sheaf, as being conjunctive for coverings is a weaker condition than simply being conjunctive. More precisely, a presheaf $P$ is \emph{conjunctive} if for any collection of open sets $\{U_a\}$ and any collection of sections $s_a\in P(U_a)$ such that the  sections agree on the overlaps, there is a section  $s\in P(\cup U_a)$ such that $s$ restricts to $s_a$ in $U_a$. A presheaf on $X$ is \emph{conjunctive for coverings} if this property holds when $\cup U_a=X$. The presheaf of singular chains $U\to S_*(X,X-\bar U)$ on $X$ is an example of a presheaf that is conjunctive for coverings, but not conjunctive; see \cite[Exercise I.12]{BR}. However, the condition of being conjunctive for coverings is sufficient for many purposes. For example, if $P$ is a monopresheaf on $X$ that is conjunctive for coverings, $\ms P$ is its sheafification, and $\Phi$ is a paracompactifying family of supports, then by \cite[Theorem I.6.2]{BR}, $P_\Phi(X)\cong \Gamma(X;\mc P)$. This accounts for the upper right vertical isomorphism in the diagram as the global presheaf section with compact supports of the conjunctive-for-coverings monopresheaf generating $\mc I^{Q_{\bar p,\bar q}} \mc S^*$ is $I^{Q_{\bar p,\bar q}}(X\times Y)$. 

We also remind the reader here that the condition of a sheaf complex $\ms S^*$  being a homotopically fine is sufficient to guarantee that $\H_\Phi^*(X;\ms S^*)\cong H^*(\Gamma_\Phi(X;\ms S^*))$. This accounts for the lower right vertical isomorphism in the diagram. See \cite{GBF20, GBF10, BR} for more details.  
}

 In fact, we know abstractly that 
$H_{n+m-*}(I^{\bar p}C_*(X)\otimes I^{\bar q}C_*(Y))\cong \H^*(\Gamma_c(X\times Y;\mc I^*))$ by \cite[Corollary 4.2]{GBF20}. But we need slightly more; we must show that the isomorphism is consistent with the left hand composition of the diagram here. 

Let $\mc K^*_X$ and $\mc K^*_Y$ be injective resolutions of $\mc I^{\bar p}\mc S^*_X$ and $\mc I^{\bar q}\mc S^*_Y$, respectively. Then we have a diagram
\begin{diagram}
H_{n+m-*}(I^{\bar p}C_*(X)\otimes I^{\bar q}C_*(Y))&\rEquals&H_{n+m-*}(I^{\bar p}C_*(X)\otimes I^{\bar q}C_*(Y))\\
\dTo^\cong &&\dTo_\cong\\
H^*(\Gamma_c(X;\mc I^{\bar p}\mc S^*_X)\otimes \Gamma_c(Y;\mc I^{\bar q}\mc S^*_Y))&\rTo^\cong &H^*(\Gamma_c(X;\mc K^*_X)\otimes \Gamma_c(Y;\mc K^*_Y))\\
\dTo&&\dTo_\cong \\
H^*(\Gamma_c(X\times Y;\pi_X^*(\mc I^{\bar p}\mc S^*)\otimes\pi_Y^*(\mc I^{\bar q}\mc S^*)))&\rTo&H^*(\Gamma_c(X\times Y;\pi_X^*(\mc K^*_X)\otimes\pi_Y^*(\mc K^*_Y))).\\
\end{diagram}
The top left vertical map is induced by sheafification and is an isomorphism because $\mc I^{\bar p}\mc S^*_X$ and $\mc I^{\bar q}\mc S^*_Y$ are induced by appropriate monopresheaves that are conjunctive for covers and whose respective global presheaf sections with compact supports are $I^{\bar p}C_*(X)$ and $I^{\bar q}C_*(Y)$. The middle and bottom horizontal maps are induced by the injective resolutions
$\mc I^{\bar p}\mc S^*_X\to \mc K^*_X$ and $\mc I^{\bar p}\mc S^*_X\to \mc K^*_Y$.
The middle horizontal map  is an isomorphism because $\mc I^{\bar p}\mc S^*_X$ and $\mc I^{\bar q}\mc S^*_Y$ are homotopically fine. We fill in the top right vertical arrow so that the top square commutes by definition. The bottom vertical maps are defined and the bottom square commutes in the evident way (given a germ over $x\in X$ and a germ over $y\in Y$, this determines a germ of the tensor product of stalks over $(x,y)$; see \cite[Section II.15]{BR}). The composition of maps on the left is equivalent to the map $H_{n+m-*}(I^{\bar p}C_*(X)\otimes I^{\bar q}C_*(Y))\to H^*(\Gamma_c(X\times Y;\mc R^*))$ of Diagram \eqref{D: crossproduct}. Since $\mc K^*_X$ and $\mc K^*_Y$ are injective, and hence $c$-fine, and since $X$ and $Y$ are locally compact and Hausdorff, $\pi_X^*(\mc K^*_X)\otimes\pi_Y^*(\mc K^*_Y)$ is $c$-fine by \cite[Exercise II.14 and page 494, fact (s)]{BR}. So the bottom map of the diagram is in fact induced by a c-fine resolution  $\pi_X^*(\mc I^{\bar p}\mc S^*_X)\otimes\pi_Y^*(\mc I^{\bar q}\mc S^*_Y)\to \pi_X^*(\mc K^*_X)\otimes\pi_Y^*(\mc K^*_Y)$. Therefore we can let the bottom horizontal map here play the role of  the bottom left map of Diagram   \eqref{D: crossproduct}. The bottom right vertical map is an isomorphism by \cite[Proposition 15.1]{BR}. 

Note, we cannot conclude that either of the maps in the diagram not labeled as such are isomorphisms, but nonetheless, this is enough to show  the composition along the left side of Diagram \eqref{D: crossproduct} is the desired isomorphism. 
\end{proof}

Next we prove the relative K\"unneth theorem (Proposition \ref{P: relative
kunneth}).  We restate it for the convenience of the reader.

\medskip

\noindent
{\bf Theorem}\ 
{\em
Let $X$ and $Y$ be stratified pseudomanifolds with open subsets $A\subset X, B\subset Y$.
The cross
product induces an isomorphism
\[
I^{\bar p}H_*(X,A;F)\otimes I^{\bar q}H_*(Y,B;F)\to I^{Q_{\bar p,\bar q}}H_*(X\times Y,(A\times Y)\cup
(X\times B);F).
\]
}

\begin{proof}
Let $Q$ denote $Q_{\bar p,\bar q}$.  Consider the following diagram (where 
we leave the $F$ coefficients tacit).
\begin{diagram}
&\rTo & I^{\bar p}H_*(A) \otimes I^{\bar q}H_*(Y)&\rTo & I^{\bar p}H_*(X) \otimes I^{\bar q}H_*(Y) &\rTo 
& I^{\bar p}H_*(X,A) \otimes I^{\bar q}H_*(Y)&\rTo&\\
&&\dTo_\times &&\dTo_\times &&\dTo_\times\\
&\rTo & I^{Q}H_*(A\times Y)&\rTo & I^{Q}H_*(X\times Y) &\rTo & I^{Q}H_*(X\times Y,A\times Y)&\rTo&.
\end{diagram}
Both rows are exact; the top row is exact because we work over a field (so all
modules are flat). The vertical maps are all induced by the chain cross
product, and the diagram commutes up to sign (as can be seen by working with
representative chains). So we have 
$I^{\bar p}H_*(X,A) \otimes I^{\bar q}H_*(Y)\cong I^{Q}H_*(X\times Y,A\times Y)$ by the five lemma.

Similarly, we now have the diagram
{\footnotesize
\begin{diagram}
&\rTo & I^{\bar p}H_*(X,A) \otimes I^{\bar q}H_*(B)&\rTo & I^{\bar p}H_*(X,A) \otimes I^{\bar q}H_*(Y) &\rTo & I^{\bar p}H_*(X,A) \otimes I^{\bar q}H_*(Y,B)&\rTo&\\
&&\dTo_\times &&\dTo_\times &&\dTo_\times\\
&\rTo & I^{Q}H_*(X\times B,A\times B)&\rTo & I^{Q}H_*(X\times Y, A\times Y) 
&\rTo & I^{Q}H_*(X\times Y,(A\times Y)\cup(X\times B))&\rTo&.
\end{diagram}}
The top row is again exact by flatness. The bottom row is the long exact sequence associated to the short exact sequence 
{\footnotesize
\begin{diagram}
0&\rTo & I^{Q}C_*(X\times B,A\times B)&\rTo & I^{Q}C_*(X\times Y, A\times Y) 
&\rTo & I^{Q_{\bar p,\bar q}}C_*(X\times Y,(A\times Y)\cup(X\times B))&\rTo 
&0,
\end{diagram}}
which exists by some basic homological algebra.\footnote{This essentially 
comes from the intersection chain short exact Mayer-Vietoris sequence for the 
pairs $(X\times Y, A\times Y)$ and $(X\times B,X\times B)$, since $B\cap 
Y=B$, $Y\cup B=Y$, $A\times Y\cap X\times B=A\times B$.} Again, commutativity 
follows from chain arguments, and the proposition now follows from the five 
lemma.
\end{proof}

\section{Invariance of general perversity intersection homology under normalization}\label{A: normal}

We provide here a theorem stating that intersection homology is preserved under normalization. For general background on normalizations see \cite{Pa03, GM1}.

\begin{lemma}\label{L: normal}
Let $X$ be a stratified pseudomanifold, and let $\pi:\hat X\to X$ be its 
normalization. Then $\pi: I^{\bar p}H_*(\hat X;R)\to I^{\bar p}H_*(X;R)$ is 
an isomorphism.
\end{lemma}
\begin{proof}
This is a standard fact for intersection homology with Goresky-MacPherson 
perversities and no codimension one strata. We briefly revisit the proof to
show that it remains true in the more general setting. It is elementary to
observe that $\pi$ is well-defined as a homomorphism of intersection chains,
and hence of intersection homology groups. The normalization map is proper
(since all stratified pseudomanifolds have compact links by definition), so we can consider intersection homology either with closed or with compact supports.

By \cite[Lemma 2.4]{GBF23}, it is sufficient to consider perversities such that $\bar p(Z)\leq \codim(Z)-1$ for each singular stratum $Z$, for otherwise we get nothing new. This fact allows us mostly to reduce the proof to the usual one: if $\bar p(Z)\leq \codim(Z)-1$ for each singular $Z$, each simplex of each allowable chain $\xi$ of $I^{\bar p}C_i(X;R)$ intersects $X^{n-1}$ in at most the image of the the $i-1$ skeleton of the model simplex $\Delta^i$. So for any such singular simplex $\sigma$ in $\xi$, $\sigma$ maps the interior of $\Delta^i$ into $X-X^{n-1}$. But this mapping of the interior can be lifted to $\hat X$, and continuity ensures that we can then lift all of $\sigma$ to $\hat X$. This process generates a homomorphism $s: I^{\bar p}C_*(X;R)\to 
I^{\bar p}C_*(\hat X;R)$, and it is clear that $s$ is an inverse of $\pi$. It only remains to check that $s$ is a chain map. This is not difficult to see, recalling that  any boundary simplices with support entirely in $X^{n-1}$ are set automatically to $0$. 
\end{proof}

\section{Comparison with the cup product of \cite{Ba10a}}
\label{l6}

In this appendix we verify the claim in Remark \ref{l5}.

First observe that for pairs $\bar p,\bar q$ satisfying the conditions in
Remark \ref{l5} we have $D\bar q\geq D\bar p+D\bar p$, so Definition \ref{l7}
gives a cup product map
\[
I_{\bar p}H^*(X;\mathbb Q)
\otimes
I_{\bar p}H^*(X;\mathbb Q)
\to
I_{\bar q}H^*(X;\mathbb Q)
\]
which we will show agrees up to sign with that constructed in 
\cite[Section 7]{Ba10a}. 

One of the ingredients in Banagl's construction is the ``Eilenberg-Zilber type
isomorphism''
\[
I^{\bar p}C_*(X;\mathbb Q)
\otimes
I^{\bar p}C_*(Y;\mathbb Q)
\to
I^{\bar p}C_*(X\times Y;\mathbb Q)
\]
((23) on page 175 of \cite{Ba10a}).\footnote{Note that there is an implicit
assumption in the construction of this map that $X$ and $Y$ are orientable, 
since the orientation sheaf mentioned on line $-12$ of \cite[page 163]{Ba10a} 
is identified on page 175 of \cite{Ba10a} with the constant sheaf 
$\mathbb{Q}_U$ (resp., $\mathbb{Q}_V$).
We leave it to the interested reader to work out the details in the
non-orientable case.
}
We will denote this map by $E$.
The criterion given at the end of \cite[Section 4.1]{GBF20} shows that, since
$\bar{p}(k)+\bar{p}(l)\leq \bar{p}(k+l)$, the cross product also
induces a map 
\[
\times: I^{\bar p}C_*(X;\mathbb Q)
\otimes
I^{\bar p}C_*(Y;\mathbb Q)
\to
I^{\bar p}C_*(X\times Y;\mathbb Q),
\]
and we claim that (up to sign) this is the same as $E$.  This follows from the
uniqueness result \cite[Proposition 2]{CGJ},
using the fact that both $E$ and $\times$ are induced by maps of sheaves 
\[
\pi_X^*\mc I^{\bar p}\mc S^*_X\otimes \pi_Y^*\mc I^{\bar p}\mc S^*_Y\to\mc
I^{\bar p}\mc S^*_{X\times Y}
\]
(see 
the proofs of \cite[Theorem 9.1]{Ba10a} and Proposition \ref{l9}) which agree
(up to sign) on $\pi^*_U\mathbb{Q}_U\otimes \pi^*_V\mathbb{Q}_V$.

Now consider the following diagram.
\begin{equation}
\label{l8}
\xymatrix{
I^{\bar q}C_*(X;\mathbb Q)
\ar[r]^-d
\ar[rd]^d
&
I^{\bar p}C_*(X\times X;\mathbb Q)
\ar[d]
&
I^{\bar p}C_*(X;\mathbb Q)
\otimes
I^{\bar p}C_*(X;\mathbb Q)
\ar[l]^-\cong_-E
\ar[ld]^\cong_\times
\\
&
I^{Q_{\bar p,\bar p}}C_*(X\times X;\mathbb Q)
&
}
\end{equation}
Here the two maps marked $d$ are induced by the diagonal; the horizontal $d$ 
is given by \cite[Proposition 7.1]{Ba10a} and the other $d$ is given by 
Proposition \ref{L: diag}.1.  The vertical map exists because of the inequality
$\bar{p}(k+l)\leq\bar{p}(k)+\bar{p}(l)+2$.  
The left-hand triangle in
diagram \eqref{l8} obviously commutes and we have just seen that the right-hand
triangle commutes up to sign.

The dual of the lower composite in diagram \eqref{l8} is
the cup product of Definition \ref{l7}, so it suffices to show that the dual of
the upper composite is the cup product of \cite[Section 7]{Ba10a}.  This in turn
is a straightforward consequence of the definition in \cite{Ba10a} and
Proposition IV.2.5 of \cite{EKMM}.

\providecommand{\bysame}{\leavevmode\hbox to3em{\hrulefill}\thinspace}
\providecommand{\MR}{\relax\ifhmode\unskip\space\fi MR }
\providecommand{\MRhref}[2]{%
  \href{http://www.ams.org/mathscinet-getitem?mr=#1}{#2}
}
\providecommand{\href}[2]{#2}

Several diagrams in this paper were typeset using the \TeX\, commutative
diagrams package by Paul Taylor. 


\begin{thebibliography}{10}

\bibitem{ALMP-combo}
Pierre Albin, Eric Leichtnam, Rafe Mazzeo, and Paolo Piazza, \emph{The
  signature package on {W}itt spaces}, To appear in Annales scientifiques de
  l'{\'E}cole normale sup{\'e}rieure; see http://arxiv.org/abs/1112.0989.

\bibitem{BaIH}
Markus Banagl, \emph{Topological invariants of stratified spaces}, Springer
  Monographs in Mathematics, Springer-Verlag, New York, 2006.

\bibitem{Ba10a}
Markus Banagl, \emph{Rational generalized intersection homology theories},
  Homology, Homotopy Appl. \textbf{12} (2010), no.~1, 157--185.

\bibitem{Bo}
A.~Borel~et. al., \emph{Intersection cohomology}, Progress in Mathematics,
  vol.~50, Birkhauser, Boston, 1984.

\bibitem{BHS}
J.P. Brasselet, G.~Hector, and M.~Saralegi, \emph{The\'eor\`eme de de{R}ham
  pour les vari\'et\'es stratifi\'ees}, Ann. Global Anal. Geom. \textbf{9}
  (1991), 211--243.

\bibitem{BR}
Glen Bredon, \emph{Sheaf theory}, Springer-Verlag, New York, 1997.

\bibitem{BRS}
S.~Buoncristiano, C.~P. Rourke, and B.~J. Sanderson, \emph{A geometric approach
  to homology theory}, Cambridge University Press, Cambridge, 1976, London
  Mathematical Society Lecture Note Series, No. 18.

\bibitem{CS91}
Sylvain~E. Cappell and Julius~L. Shaneson, \emph{Singular spaces,
  characteristic classes, and intersection homology}, Annals of Mathematics
  \textbf{134} (1991), 325--374.

\bibitem{CST}
David Chataur, Martintxo Saralegi-Aranguren, and Daniel Tanr{\'e},
  \emph{Rational homotopy and intersection cohomology}, Preprint available at
  arxiv.org/pdf/1205.7057.

\bibitem{CGJ}
Daniel~C. Cohen, Mark Goresky, and Lizhen Ji, \emph{On the {K}{\"u}nneth
  formula for intersection cohomology}, Trans. Amer. Math. Soc. \textbf{333}
  (1992), 63--69.

\bibitem{Dold}
Albrecht Dold, \emph{Lectures on algebraic topology}, Springer-Verlag,
  Berlin-Heidelberg-New York, 1972.

\bibitem{EKMM}
A.~D. Elmendorf, I.~Kriz, M.~A. Mandell, and J.~P. May, \emph{Rings, modules,
  and algebras in stable homotopy theory}, Mathematical Surveys and Monographs,
  vol.~47, American Mathematical Society, Providence, RI, 1997, With an
  appendix by M. Cole.

\bibitem{GBF35}
Greg Friedman, \emph{An introduction to intersection homology without sheaves},
  book in preparation.

\bibitem{GBF3}
\bysame, \emph{Stratified fibrations and the intersection homology of the
  regular neighborhoods of bottom strata}, Topology Appl. \textbf{134} (2003),
  69--109.

\bibitem{GBF13}
\bysame, \emph{Intersection homology of stratified fibrations and
  neighborhoods}, Adv. Math. \textbf{215} (2007), no.~1, 24--65.

\bibitem{GBF10}
\bysame, \emph{Singular chain intersection homology for traditional and
  super-perversities}, Trans. Amer. Math. Soc. \textbf{359} (2007), 1977--2019.

\bibitem{GBF20}
\bysame, \emph{Intersection homology {K}\"unneth theorems}, Math. Ann.
  \textbf{343} (2009), no.~2, 371--395.

\bibitem{GBF18}
\bysame, \emph{On the chain-level intersection pairing for {PL}
  pseudomanifolds}, Homology, Homotopy and Applications \textbf{11} (2009),
  261--314.

\bibitem{GBF23}
\bysame, \emph{Intersection homology with general perversities}, Geometriae
  Dedicata \textbf{148} (2010), 103--135.

\bibitem{GBF26}
\bysame, \emph{An introduction to intersection homology with general perversity
  functions}, Topology of Stratified Spaces, Mathematical Sciences Research
  Institute Publications, vol.~58, Cambridge University Press, 2011,
  pp.~177--222.

\bibitem{GBF31}
Greg Friedman and James McClure, \emph{The symmetric signature of a {W}itt
  space}, see http://arxiv.org/abs/1106.4798.

\bibitem{GBF30}
\bysame, \emph{Verdier duality and the cap product, for manifolds and
  pseudomanifolds}, in preparation.

\bibitem{GM1}
Mark Goresky and Robert MacPherson, \emph{Intersection homology theory},
  Topology \textbf{19} (1980), 135--162.

\bibitem{GM2}
\bysame, \emph{Intersection homology {II}}, Invent. Math. \textbf{72} (1983),
  77--129.

\bibitem{Go81}
R.~Mark Goresky, \emph{Whitney stratified chains and cochains}, Trans. Amer.
  Math. Soc. \textbf{267} (1981), no.~1, 175--196.

\bibitem{HS91}
Nathan Habegger and Leslie Saper, \emph{Intersection cohomology of cs-spaces
  and {Z}eeman's filtration}, Invent. Math. \textbf{105} (1991), 247--272.

\bibitem{Ha}
Allen Hatcher, \emph{Algebraic topology}, Cambridge University Press,
  Cambridge, 2002.

\bibitem{Ki}
Henry~C. King, \emph{Topological invariance of intersection homology without
  sheaves}, Topology Appl. \textbf{20} (1985), 149--160.

\bibitem{KirWoo}
Frances Kirwan and Jonathan Woolf, \emph{An introduction to intersection
  homology theory. second edition}, Chapman \& Hall/CRC, Boca Raton, FL, 2006.

\bibitem{Pa03}
G.~Padilla, \emph{On normal stratified pseudomanifolds}, Extracta Math.
  \textbf{18} (2003), no.~2, 223--234.

\bibitem{Q1}
Frank Quinn, \emph{Homotopically stratified sets}, J. Amer. Math. Soc.
  \textbf{1} (1988), 441--499.

\bibitem{Sa05}
Martintxo Saralegi-Aranguren, \emph{de {R}ham intersection cohomology for
  general perversities}, Illinois J. Math. \textbf{49} (2005), no.~3, 737--758
  (electronic).

\end{thebibliography}
\end{document}